\documentclass[12pt,reqno,a4paper]{amsart}
\usepackage{
    amsmath, amsfonts, amssymb,  amsthm,   amscd, 
    gensymb,  graphicx, comment,  etoolbox, url,
    booktabs, stackrel, mathtools,enumitem, mathdots,  microtype, lmodern,    mathrsfs, graphicx, tikz,  longtable,tabularx, float, tikz, pst-node, tikz-cd, multirow, tabularx, amscd,  bm, array, makecell, diagbox, booktabs,ragged2e, caption, subcaption }
\usepackage{makecell,slashbox}
\usepackage{thmtools} 

\usepackage{xcolor}
\usepackage[utf8]{inputenc}
\usepackage{microtype, fullpage, wrapfig,textcomp,mathrsfs,csquotes,fbb}
\usepackage[pagebackref=true, colorlinks=true, linkcolor=blue, citecolor=blue, urlcolor=blue, breaklinks=true]{hyperref}

\usepackage[capitalise]{cleveref}

\usepackage{todonotes}
\usetikzlibrary{positioning}
\usetikzlibrary{shapes,arrows.meta,calc}
\usetikzlibrary{arrows}

\newtheorem{theorem}{Theorem}[section]

\newtheorem{conjecture}[theorem]{Conjecture}
\newtheorem{corollary}[theorem] {Corollary}
\newtheorem{lemma}[theorem]{Lemma}

\newtheorem{proposition}[theorem]{Proposition}

\newtheorem*{theorem*}{Theorem}
\newtheorem*{conjecture*}{Conjecture}

\theoremstyle{definition}
\newtheorem{definition}[theorem]{Definition}
\newtheorem{example}[theorem]{Example}
\newtheorem{remark}[theorem]{Remark}

\newtheorem{question}[theorem]{Question}

\newtheorem*{notation*}{Notation}

\newcommand{\R}{\mathbb{R}}
\newcommand{\Z}{\mathbb{Z}}
\newcommand{\C}{\mathbb{C}}

\newcommand{\TC}{\mathrm{TC}}

\newcommand{\ct}{\mathrm{cat}}

\newcommand{\sct}{\mathrm{secat}}
\newcommand{\rct}{\mathrm{relcat}}
\newcommand{\pr}{\mathrm{pr}}
\newcommand{\id}{\mathrm{id}}
\newcommand{\op}{\mathrm{op}}
\newcommand{\hdim}{\mathrm{hdim}}

\makeatletter
\def\@tocline#1#2#3#4#5#6#7{\relax
  \ifnum #1>\c@tocdepth 
  \else
    \par \addpenalty\@secpenalty\addvspace{#2}%
    \begingroup \hyphenpenalty\@M
    \@ifempty{#4}{%
      \@tempdima\csname r@tocindent\number#1\endcsname\relax
    }{%
      \@tempdima#4\relax
    }%
    \parindent\z@ \leftskip#3\relax \advance\leftskip\@tempdima\relax
    \rightskip\@pnumwidth plus4em \parfillskip-\@pnumwidth
    #5\leavevmode\hskip-\@tempdima
      \ifcase #1
       \or\or \hskip 1em \or \hskip 2em \else \hskip 3em \fi%
      #6\nobreak\relax
    \hfill\hbox to\@pnumwidth{\@tocpagenum{#7}}\par
    \nobreak
    \endgroup
  \fi}
\makeatother

\newcolumntype{x}[1]{>{\centering\arraybackslash}p{#1}}

\begin{document}
\title[]{Monoidal and symmetrized \\ parametrized topological complexity}
\author[R. Singh]{Ramandeep Singh Arora}
\address{Department of Mathematics, Indian Institute of Science Education and Research Pune, India}
\email{ramandeepsingh.arora@students.iiserpune.ac.in}
\email{ramandsa@gmail.com}
\author[N. Daundkar]{Navnath Daundkar}
\address{Department of Mathematics, Indian Institute of Technology Madras, Chennai, India.}
\email{navnath@iitm.ac.in}

\thanks{}

\begin{abstract} 
We introduce and study monoidal and symmetrized versions of parametrized topological complexity.
First, we develop parametrized analogues of the monoidal topological complexity theories of Iwase--Sakai, Aguilar-Guzm\'an--Gonz\'alez (based on the Fadell--Husseini approach), and Dranishnikov. 
We investigate the parametrized Iwase--Sakai conjecture and provide sufficient conditions under which it holds.
We then introduce symmetrized parametrized topological complexity and combine it with the monoidal perspective to define monoidal symmetrized parametrized topological complexity, showing that the resulting notions agree.
Finally, we compute these invariants for Fadell--Neuwirth fibrations.
\end{abstract}

\maketitle

\tableofcontents

\section{Introduction}\label{sec:intro}
The \emph{topological complexity} of a space $X$, denoted by $\TC(X)$, is the smallest positive integer $k$ for which the product space $X \times X$ admits an open cover ${U_0,\ldots,U_k}$ such that each $U_i$ admits a continuous section of the free path fibration
\begin{equation}
\label{eq: free path space fibration}
    \pi \colon X^I \to X \times X, \qquad \pi(\gamma)=(\gamma(0),\gamma(1)),
\end{equation}
where $X^I$ is the free path space of $X$ endowed with the compact-open topology.
This invariant was introduced by Farber in \cite{FarberTC} to study the complexity of motion planning algorithms for a mechanical system whose configuration space is $X$. 
Since then, topological complexity has emerged as one of the fundamental numerical homotopy invariants in topological robotics and has been studied extensively from both theoretical and computational perspectives.

In many practical applications, it is natural to require motion planners to satisfy additional compatibility conditions beyond merely assigning continuous paths. 
For example, if the initial and terminal configurations coincide, the assigned motion should remain stationary throughout the entire interval. 
Likewise, reversing the roles of the initial and terminal configurations should reverse the direction of the prescribed path, reflecting the inherent symmetry of the motion-planning problem. 
Imposing such natural constraints leads to refined notions of topological complexities, namely the symmetric topological complexity $\TC^S$, monoidal topological complexity $\TC^M$, and symmetrized topological complexity $\TC^{\Sigma}$, which have been studied extensively in the literature \cite{FG07, I-S, gonzalezhighertc, Grantsymmtc}.

Further, the system itself can be influenced by external parameters that may vary over time or from one situation to another. 
Consequently, it is desirable to study motion planners that adapt continuously to these varying conditions rather than treating each instance independently. 
This viewpoint leads to the notion of parametrized topological complexity $\TC[p \colon E \to B]$ of a fibration $p \colon E \to B$, introduced by Cohen, Farber and Weinberger in \cite{farber-para-tc}, which extends the classical theory by incorporating families of motion-planning problems into a unified framework. 
It provides a natural setting for investigating how changes in external parameters affect the complexity of motion planning.

Parametrized topological complexity has been the subject of considerable recent interest. We refer the reader to \cite{farber-para-tc, PTCcolfree, ptcspherebundles, minowa2024parametrized} for its foundational properties and computational aspects. 
The notion was subsequently extended to the fibrewise setting by Garc\'{\i}a-Calcines \cite{fibrewise}, while Crabb \cite{crabb2023fibrewise} established several computational results in this more general framework.
More recently, the second author extended parametrized topological complexity to the equivariant setting \cite{D-EqPTC}, and the first and second authors introduced the notion of invariant parametrized topological complexity \cite{Inv-Para-TC}.

The primary aim of this paper is to develop the monoidal and symmetrized variants of topological complexity in the parametrized setting. 
We introduce these new parametrized invariants, establish their basic properties, and investigate the interplay between the monoidal and symmetry conditions and the parametrized framework. 
Furthermore, we extend several fundamental results from the classical theory to this more general setting, thereby enriching the theory of parametrized motion planning.

\subsection*{Outline of the paper}

\Cref{sec: preliminaries} establishes the notation and conventions used throughout the paper and reviews the necessary background. 
We also obtain some new results in \Cref{thm: restriction map is a G-fibration} and \Cref{prop: cofib E to E times B E}.
We further show in \Cref{thm: dim--hdim in Dranishnikov's result} that Dranishnikov's result \cite[Theorem 2.5]{dranishnikov2014topological} on monoidal topological complexity can be strengthened by replacing dimension with homotopy dimension.
More precisely,

\begin{theorem*}
Suppose $X$ is a $k$-connected CW complex.
If $\mathrm{hdim}(X) < (k+1)(\TC(X)+1)-1$, then
$$
    \TC^{M}(X) = \TC(X).
$$
\end{theorem*}

As a consequence, we show that the monoidal topological complexity of configuration spaces matches that of their usual topological complexity, see \Cref{ex: monoidal TC of conf spaces}.

In \Cref{sec: mono-para-tc}, we introduce and study the monoidal parametrized topological complexity $\TC^{M}[p \colon E \to B]$ of a fibration $p \colon E \to B$, including its formulation, fundamental properties, and connections with existing notions. 
For example, in \Cref{thm: equivalent defn of monoidal para tc}, we show that the monoidal parametrized topological complexity is a relative category (\Cref{defn: relative category}). 
In \Cref{subsec: para-Iwase-Sakai-conj}, we formulate the parametrized version of the Iwase--Sakai conjecture in the following sense: 

\begin{conjecture}(Parametrized I-S conjecture)
\label{conj: para IS conjecture}
If $p \colon E \to B$ is a locally trivial fibration, where $E$ and $B$ are locally finite simplicial complexes, then
$$
    \TC^{M}[p \colon E \to B] = \TC[p \colon E \to B]\,.
$$
\end{conjecture}

We verify this conjecture in several cases.
In particular, in \Cref{thm: mon-para-tc = para-tc if dim(E) < (TC[p] + 1)(k+1)-1} we show the following:

\begin{theorem*}
Suppose $p \colon E \to B$ is an LEC fibration with the fibre $F$, where $E$ is a CW complex and $E \times_B E$ is paracompact.
If $F$ is $k$-connected such that
$$
    \dim(E) < (\TC[p \colon E \to B]+1) \cdot (k+1) - 1, 
$$
then $\TC^{M}[p \colon E \to B] = \TC[p \colon E \to B]$.    
\end{theorem*}

In \Cref{prop: Iwase-Sakai conjecture for principal G-bundle}, we show that for a principal $G$-bundle $p \colon E \to B$, where $G$ is a connected Lie group, the following equalities hold:
$$
    \TC^{M}[p \colon E \to B] 
        = \TC[p \colon E \to B] 
        = \TC^{M}(G) 
        = \TC(G) 
        = \ct(G).
$$
We further establish \Cref{thm: mon-para-tc = para-tc under special cover}, \Cref{cor: mon-para-tc = para-tc under gen-special cover}, and \Cref{cor: mon-para-tc = para-tc of associated bundle}, which provide conditions on covers of various spaces under which the conjecture holds, allowing us to compute several examples.
This section also develops a $D$-monoidal version (\Cref{defn: D-mon-para-tc}) and a Fadell--Husseini version (\Cref{defn: FH-mon-para-tc}) of monoidal parametrized topological complexity, as well as their generalized versions.
In particular, we obtain the following results in \Cref{thm: mon-para-tc = D-mon-para-tc} and \Cref{thm: mon-para-tc = FH-mon-para-tc}, respectively:

\begin{theorem*}
Suppose $p \colon E \to B$ is a fibration, where $E$ and $E \times_B E$ are ANRs.
Then
$$
    \TC^{M}[p \colon E\to B]
        = \TC_g^{M}[p \colon E\to B]
        = \TC^{DM}[p \colon E\to B].
$$
\end{theorem*}

\begin{theorem*}
Suppose $p \colon E \to B$ is a fibration, where $E$ and $E \times_B E$ are ANRs.
Then
$$
    \TC_g^{FH}[p \colon E\to B]
        = \TC^{FH}[p \colon E \to B]
        = \TC^{M}[p \colon E\to B].
$$
\end{theorem*}

\Cref{sec: symm-para-tc} is devoted to developing the notion of symmetrized parametrized topological complexity $\TC^{\Sigma}[p \colon E \to B]$ of a fibration $p \colon E \to B$. 
We establish its main properties such as fibrewise homotopy invariance (\Cref{prop: fibrewise homotopy invariance of sym-para-tc}), and examine its relationship with other topological complexity invariants in \Cref{prop: sym-para-tc of projection map} and \Cref{prop: para-tc leq sym-para-tc}.
We also provide the following lower and upper bounds for it in \Cref{thm: cohomological lower bound on sym-para-tc} and \Cref{thm: dimension-connectivity upper bound on symmetrized TC}, respectively.

\begin{theorem*}
Suppose $p \colon E \to B$ is a fibration. 
Let $SP^2_B(E) := (E\times_B E)/G$ be the fibrewise symmetric square of $E$, let $\rho \colon E\times_B E \to SP^2_B(E)$ be the orbit map, and let $dE_B := \rho(\Delta E)$ denote the image of the diagonal under $\rho$.
If there exists cohomology classes $u_1,\dots,u_k \in H^*(SP^2_B(E);R)$ (for any commutative ring $R$) such that
    \begin{enumerate}
        \item $u_i$ restricts to zero in $H^*(dE_B;R)$ for $i=1,\dots,k$;
        \item $u_1 \smile \dots \smile u_k \neq 0$ in $H^*(SP^2_B(E);R)$, 
    \end{enumerate}
then $\TC^{\Sigma}[p \colon E \to B]\geq k.$    
\end{theorem*}

\begin{theorem*}
Suppose $p \colon E \to B$ is a fibration with fibre $F$ such that $E \times_B E$ is a $G$-CW complex of dimension at least 2.
If $F$ is $s$-connected, then
$$
    \TC^{\Sigma}[p \colon E \to B] 
        < \frac{\hdim_G(E \times_B E)+1}{s+1}.
$$
\end{theorem*}

In \Cref{sec: mon-symm-para-tc}, we combine the monoidal and symmetrized perspectives to define monoidal symmetrized parametrized topological complexity $\TC^{M,\,\Sigma}[p \colon E \to B]$ of a fibration $p \colon E \to B$. 
Its connections with the previously introduced notions are captured in the main result of the section, stated in \Cref{thm: equivalent defn of sym monoidal para tc}:

\begin{theorem*}
Suppose $p \colon E \to B$ is a locally trivial fibration, where $E$ is an ENR and $B$ is a separable ANR.
Then the following statements are equivalent:
\begin{enumerate}
\item $\TC^{M,\,\Sigma}[p \colon E\to B] \leq k$.

\item $\TC^{FH,\,\Sigma}[p \colon E\to B] \leq k$.
        
\item $\TC^{DM,\,\Sigma}[p \colon E\to B] \leq k$.

\item $\TC^{\Sigma}[p \colon E\to B] \leq k$.
        
\item The $G$-fibration $\Pi_k \colon J^k_{E \times_B E}(E^I_B) \to E \times_B E$ admits a global $G$-section.
        
\item The $G$-fibration $\Pi_k \colon J^k_{E \times_B E}(E^I_B) \to E \times_B E$ admits a global $G$-section $\sigma$ satisfying $\sigma \circ \Delta_E = s_k$.
\end{enumerate}
\end{theorem*}

We also provide the following lower bound for it in \Cref{thm: cohomological lower bound on mon-sym-para-tc}.

\begin{theorem*}
Let $p\colon E\to B$ be a fibration. 
If there exists relative cohomology classes $v_1,\dots, v_k \in H^*(SP^2_B(E), dE_B; R)$ (for any commutative ring $R$) such that 
$$
    0 \neq v_1\smile \dots \smile v_k \in H^*(SP^2_B(E), dE_B; R)
$$
then $\TC^{M, \,\Sigma}[p:E\to B]\geq k$.    
\end{theorem*}

Finally, \Cref{sec: Fadell-Neuwirth fibrations} focuses on applications to Fadell--Neuwirth fibrations, where we compute the parametrized topological complexity invariants developed in this paper. In particular, we obtain the following result.

\begin{theorem*}
Suppose $n \geq 1$, $m \geq 2$ and $d \geq 2$. 
Then
$$
\TC^{\Sigma}[\mkern 1mu p \colon F(\R^d,n+m) \to F(\R^d,m)\mkern 1mu] =
    \begin{cases}
        2n+m-1, & \text{if $d$ is odd},\\
        \text{either }2n+m-2 \text{ or } 2n+m-1, & \text{if $d$ is even}. 
    \end{cases}
$$    
\end{theorem*}

Moreover, the monoidal parametrized topological complexity and the monoidal symmetrized parametrized topological complexity satisfy the same estimate as above.

\subsection*{Notations and conventions} 
Throughout the text, $G$ denotes a compact Lie group acting on Hausdorff spaces, unless stated otherwise. 
We adopt standard terminology and notation from equivariant topology, such as $G$-spaces, $G$-maps, $G$-homotopies and related notions.

\section{Preliminaries}\label{sec: preliminaries}

In this section, we systematically recall various numerical invariants: equivariant sectional category, relative category, monoidal topological complexity, symmetrized topological complexity and parametrized topological complexity.

\subsection{$G$-Fibrations and $G$-Cofibrations}
\label{subsec: eq-fibrations}
\hfill\\ \vspace{-0.7em}

We begin by recalling the definition of $G$-fibrations and $G$-cofibrations.
For a more detailed discussion, we refer the reader to \cite[Section 2]{Grantsymmtc}.

\begin{definition}
A $G$-map $p \colon E \to B$ is called a \emph{$G$-fibration} if it has the $G$-homotopy lifting property with respect to any $G$-space $X$.
More precisely, if $H \colon X \times I \to B$ is a $G$-homotopy and $f \colon X \to E$ is a $G$-map with $p \circ f = H_0$, then there exists a $G$-homotopy $\widetilde{H} \colon X \times I \to E$ such that the following diagram
\begin{equation}
\label{diag: fibration}
\begin{tikzcd}
X  \arrow[r, "f"] \arrow[d, "i_0"', hook]                     & E \arrow[d, "p"] \\
X \times I \arrow[r, "H"] \arrow[ru, "\widetilde{H}", dotted] & B               
\end{tikzcd}
\end{equation}
commutes.
\end{definition}

We briefly recall some basic facts about fixed points under closed subgroups, both to fix notation and for later use. 
Let $H$ be a closed subgroup of $G$, and let $X$ be a $G$-space. 
The set of $H$-fixed points of $X$, denoted by $X^H$, is defined as
\[
    X^H := \{x \in X \mid h \cdot x = x \text{ for all } h \in H\}.
\]
If $f \colon X \to Y$ is a $G$-map, then $f(X^H) \subseteq Y^H$, so the restriction of $f$ to $X^H$ induces a map 
\[
    f^H \colon X^H \to Y^H.
\]

The proof of the following proposition is immediate and left to the reader.

\begin{proposition}
\label{prop: fibre of fixed point fibration}
Suppose $p \colon E \to B$ is a $G$-fibration. 
If $b \in B^G$, then the $G$-action on $E$ restricts to a $G$-action on the fibre $F = p^{-1}(b)$.
Moreover, for each closed subgroup $H$ of $G$, the fibre over $b$ of the induced fibration $p^{H} \colon E^H \to B^H$ is $F^H$. 
\end{proposition}

\begin{lemma}
Suppose $p \colon E \to B$ is a fibration, and $G$ is a topological group.
If $E$ and $B$ are considered as $G$-spaces with the trivial $G$-action, then $p$ is a $G$-fibration.
\end{lemma}

\begin{proof}
Suppose $X$ is a $G$-space. 
Then any diagram \eqref{diag: fibration} of $G$-maps factors through the orbit spaces as:
$$
\begin{tikzcd}
X \arrow[d, "i_0"', hook] \arrow[r, "q"]       & X/G \arrow[r, "f'"] \arrow[d, "j_0"', hook] & E \arrow[d, "p"] \\
X \times I \arrow[r, "q \times \mathrm{id}_I"] & X/G \times I \arrow[r, "H'"]                & {B\,,}          
\end{tikzcd}
$$
where $H = H' \circ (q \times \id_I)$, $f = f' \circ q$, and $q$ is the orbit map.
If $\widetilde{H} \colon X/G \times I \to E$ is the lift of $H'$ making the right square commute, then $\widetilde{H} \circ (q \times \id_I)$ is a $G$-map that makes the outer diagram commute.
\end{proof}

\begin{definition}
Let $L$ be a closed $G$-subspace of a $G$-space $K$.
The inclusion map $i \colon L \hookrightarrow K$ is said to be a \emph{$G$-cofibration} if every $G$-map $f \colon (K \times \{0\}) \cup (L \times I) \to Y$ to any $G$-space $Y$ has a continuous $G$-extension $F \colon K \times I \to Y$, where $G$ acts trivially on $I$ and diagonally on $K \times I$.
\end{definition}

Suppose $p \colon E \to B$ is a map. 
For a topological space $K$, let $E^K$ be the space of all continuous maps from $K$ to $E$ equipped with compact-open topology. 
Let $E^K_B$ be the subspace of $E^K$ defined as
$$
    E^K_B := \{f \in E^K \mid (p \circ f)(k) = b ~\text{for some}~ b \in B ~\text{and for all}~ k \in K\}.
$$
If $K$ is equipped with a $G$-action, then $E^K_B$ admits $G$-action given by
$$
    (g \cdot f)(k) := f(g^{-1}k)
$$
for $g\in G, f \in E^K_B$ and $k \in K$.
Under this induced action, the following theorem generalizes \cite[Appendix]{PTCcolfree} to the equivariant setting and \cite[Proposition 2.5]{Grantsymmtc} to the parametrized setting.

\begin{theorem}
\label{thm: restriction map is a G-fibration}
Suppose $p \colon E \to B$ is a fibration. 
If $i \colon L \hookrightarrow K$ is a $G$-cofibration, then the restriction map
$$
    \Pi \colon E^K_B \to E^L_B, ~\text{given by}~\Pi(f) = \left.f\right|_{L}
$$
is a $G$-fibration. 
\end{theorem}

\begin{proof}
Let $X$ be a $G$-space, and suppose we are given a commutative diagram
\[
\begin{tikzcd}
X \arrow[d, "i_0"', hook] \arrow[r, "f"] & E^{K}_B \arrow[d, "\Pi"] \\
X \times I \arrow[r, "H"]                & {E^{L}_B\,,}            
\end{tikzcd}
\]
where $i_0$ is the inclusion map given by $i_0(x)=(x,0)$, and $f$ and $H$ are $G$-maps, with $G$ acting trivially on the unit interval $I$. 
Since $f$ is $G$-equivariant, we have 
$$ 
    f(g x)(k) = (g \cdot f(x))(k) = f(x)(g^{-1}k)
$$
for all $x \in X$ and $k \in K$.
Similarly, we have $H(gx,t)(l) = H(x,t)(g^{-1}l)$ for all $(x,l,t) \in X \times L \times I$.

\vspace{0.3 em}
Define a map $F \colon (X \times K \times \{0\}) \,\cup\,  (X \times L \times I) \to E$ by
$$
F(x,k,t) :=
    \begin{cases}
        f(x)(k), & \text{if } (x,k,t) \in X \times K \times \{0\},\\
        H(x,t)(k), & \text{if } (x,k,t) \in X \times L \times I.
    \end{cases}
$$
This is well defined, since $f(x)(l) = H(x,0)(l)$ for all $(x,l,0) \in X \times L \times \{0\}$. 
Moreover, $F$ is $G$-equivariant, where $X \times K \times \{0\}$ and $X \times L \times I$ are equipped with the diagonal $G$-action.

\vspace{0.3 em}
Define a map $H' \colon X \times K \times I \to B$ as follows: given $(x,k,t) \in X \times K \times I$, the map $p \circ H(x,t) \colon L \to B$ is a constant map, say given by the value $b_{x,t} \in B$. 
Then define $H'(x,k,t) = b_{x,t}$.
Note that $H'$ is $G$-equivariant, since $b_{gx,t} = H(gx,t)(l) = H(x,t)(g^{-1}l) = b_{x,t}$.
Observe that $p \circ f(x) \colon K \to B$ is also a constant map, say given by the value $b_{x} \in B$.
Then $b_{x,0}=b_x$, since $\left.f(x)\right|_{L} = H(x,0)$.

\vspace{0.3 em}
Consider the following diagram
\[
\begin{tikzcd}
{(X \times K \times \{0\}) \,\cup\,  (X \times L \times I)} \arrow[rr, "F"] \arrow[d, "j"', hook] &  & E \arrow[d, "p"] \\
X \times K \times I \arrow[rr, "H'"]                                                              &  & {B\,.}          
\end{tikzcd}
\]
Then the above diagram commutes, since $p (F(x,k,0)) = p(f(x)(k)) = (p \circ f(x))(k) = b_x = b_{x,0} = H'(x,k,0)$ for $(x,k,0) \in X \times K \times \{0\}$, and $p (F(x,l,t)) = p(H(x,t)(l)) = (p \circ H(x,t))(l) = b_{x,t} = H'(x,l,t)$ for $(x,l,t) \in X \times L \times I$.
Since the inclusion $X \times L \hookrightarrow X \times K$ is a $G$-cofibration, by \cite[Proposition 2.9]{Grantsymmtc}, it follows that there exists a $G$-map
$$
    \widehat{H} \colon X\times K\times I \to E
$$
such that $p \circ \widehat{H} = H'$ and $\widehat{H} \circ j = F$.

\vspace{0.3 em}
Let $\widetilde{H} \colon X \times I \to E^K$ be the map defined as
$
    \widetilde H(x,t)(k) := \widehat{H}(x,k,t).
$
Then $(p \circ \widetilde{H}(x,t)) \colon K \to B$ is the constant map given by the value $b_{x,t} \in B$, since 
$$
    (p \circ \widetilde{H}(x,t))(k)
        = p(\widetilde{H}(x,t)(k))
        = p(\widehat{H}(x,k,t))
        = H'(x,k,t)
        = b_{x,t}\,.
$$
Moreover, 
$
    (\Pi \circ \widetilde{H}(x,t))(l) 
        = \widehat{H}(x,l,t) 
        = F(x,l,t)
        = H(x,t)(l)
$
implies $\Pi \circ \widetilde{H} = H$, and 
$
    \widetilde{H}(x,0)(k)
        = \widehat{H}(x,k,0)
        = F(x,k,0)
        = f(x)(k)
$
implies $\widetilde{H} \circ i_{0} = f$. 
Since 
$$
    \widetilde{H}(gx,t)(k) 
        = \widehat{H}(gx,k,t) 
        = \widehat{H}(x,g^{-1}k,t)
        = \widetilde{H}(x,t) (g^{-1}k)
        = (g \cdot \widetilde{H}(x,t))(k),
$$
it follows that $\widetilde{H}$ is $G$-equivariant.
Hence, $\widetilde{H}$ is the desired lift and $\Pi$ is a $G$-fibration.
\end{proof}

\subsection{$G$-ENRs and $G$-ANRs}
\label{subsec: eq-enr}
\hfill\\ \vspace{-0.7em}

In this section, we recall the definition of $G$-ENRs and $G$-ANRs.
For a more detailed discussion, we refer the reader to \cite{murayama-G-ANR}.

\begin{definition}
A $G$-space $X$ is said to be a \emph{$G$-ENR} if there exists a finite-dimensional $G$-representation $V$ and a $G$-embedding $h \colon X \to V$ such that $h(X)$ is a $G$-retract of some $G$-invariant open subset of $V$.
\end{definition}

\begin{definition}[{\cite[Section 4]{murayama-G-ANR}, Compare \cite[Chapter IV, Section 1]{borsuk-retracts}}]
A $G$-space $X$ is said to be a $G$-ANR if $X$ is metrizable, and for each $G$-homeomorphism $h$ mapping $X$ onto a closed $G$-subset $h(X)$ of a metrizable $G$-space $Y$, $h(X)$ is a $G$-neighborhood retract of $Y$.
\end{definition}

\begin{proposition}[{\cite[Proposition 6.10]{murayama-G-ANR}}]
\label{prop: ENR implies ANR}
    Every $G$-ENR is a $G$-ANR.   
\end{proposition}

The following lemma is the equivariant analogue of \cite[Chapter IV, Lemma (8.2)]{borsuk-retracts} and can be proved using $G$-Urysohn's lemma (\cite[Lemma 6.12]{murayama-G-ANR}).

\begin{lemma}[{Compare \cite[Chapter IV, Lemma (8.2)]{borsuk-retracts}}]
\label{lemma: (for) inclusion of G-ANRs is a G-cofibration}
    If $A$ is a closed $G$-subset of a metrizable $G$-space $X$, then for every open $G$-neighborhood $U$ of the set 
    $
        (X \times \{0\}) \times (A \times I)
    $
    in the space $X \times I$, there exists a $G$-map $\phi \colon X \times I \to U$ which is identity on $(X \times \{0\}) \times (A \times I)$.
\end{lemma}

\begin{theorem}[{Compare \cite[Section 14, Corollary 7B]{daverman-manifolds}}]
\label{thm: cobfib}
Suppose $X$ is a $G$-ANR and $A$ is a closed $G$-subset of $X$ that is also a $G$-ANR, then the inclusion map $i \colon A \hookrightarrow X$ is a $G$-cofibration.
\end{theorem}

\begin{proof}
    By \cite[Proposition 6.8 and Corollary 8.2]{murayama-G-ANR}, it follows that $X \times I$ and $A \times I$ are $G$-ANRs, where $I$ is equipped with the trivial $G$-action. 
    In particular, $X \times I$ is a metrizable $G$-space.
    Hence, by \cite[Corollary 7.4 (2)]{murayama-G-ANR}, it follows that $(X \times \{0\}) \cup (A \times I)$ is a $G$-ANR.
    Thus, there exists an open $G$-neighborhood $U$ of $(X \times \{0\}) \times (A \times I)$ in the space $X \times I$ and a $G$-retraction $r \colon U \to (X \times \{0\}) \cup (A \times I)$.
    
    Now suppose $Y$ is a $G$-space and $f \colon (X \times \{0\}) \cup (A \times I) \to Y$ is a $G$-map.
    Then, by \Cref{lemma: (for) inclusion of G-ANRs is a G-cofibration}, there exists a $G$-map $\phi \colon X \times I \to U$ which is identity on $(X \times \{0\}) \times (A \times I)$.
    Then $F \colon X \times I \to Y$ given by 
    $$
        F(x,t) = f(r(\phi(x,t))),
    $$
    is a $G$-equivariant map extending $f$.
\end{proof}

\begin{definition} 
Let $p \colon E \to B$ be a map. 
\begin{enumerate}
    \item The map $p$ is said to be \emph{locally equiconnected}, or \emph{LEC}, if the diagonal inclusion $\Delta_E \colon E \to E \times_B E$ is a closed cofibration.

    \item The map $p$ is said to be \emph{symmetrically locally equiconnected}, or \emph{symmetrically LEC}, if the diagonal inclusion $\Delta_E \colon E \to E \times_B E$ is a closed $\Z_2$-equivariant cofibration, where $\mathbb{Z}_2$ acts trivially on $E$ and by transposition on $E \times_{B} E$.

    \item The map $p$ is said to be an \emph{LEC fibration} (respectively, a \emph{symmetrically LEC fibration}) if it is LEC (respectively, symmetrically LEC) and a fibration.
\end{enumerate}
\end{definition}

We note that if $p \colon E \to B$ is a map such that both $E$ and $E \times_B E$ are ANRs, then it follows from \Cref{thm: cobfib} that the diagonal inclusion 
$
    \Delta_E \colon E \to E \times_B E
$
is a closed cofibration; that is, the map $p$ is LEC. 
Now we state the result of Cohen, Farber and Weinberger, which gives a sufficient condition for the space $E \times_B E$ to be an ANR.

\begin{proposition}[{\cite[Proposition 4.7]{farber-para-tc}}]
\label{prop: fibre product and path space are ANR}
If $p \colon E \to B$ is a locally trivial fibration with fibre $F$, where $E$ and $B$ are metrizable separable ANRs, then $E \times_B E$ is a metrizable separable ANR.
In particular, the map $p$ is LEC.
\end{proposition}

Now we show that, if $E$ is an ENR rather than a separable ANR in \Cref{prop: fibre product and path space are ANR}, then the map $p$ is also symmetrically LEC.

\begin{proposition}
\label{prop: cofib E to E times B E}
    Suppose $p \colon E \to B$ is a locally trivial fibration, where $E$ is an ENR and $B$ is a separable ANR.
    Then $E \times_B E$ is a $\Z_2$-ENR, and the map $p$ is symmetrically LEC.
\end{proposition}

\begin{proof}
    Since $E$ is a neighborhood retract of a Euclidean space, it follows by \cite[Chapter IV, Proposition 8.1 and Lemma 8.3]{dold-alg-top} that $E$ is locally compact.
    Hence, by \cite[Page 186, Exercise 2]{munkres2000topology}, $E \times E$ is locally compact.
    As $B$ is Hausdorff, $E \times_B E$ is a closed subset of $E \times E$, and hence, by \cite[Corollary 29.3]{munkres2000topology}, it is also locally compact.
    Furthermore, since $E \times E$ is also an ENR, it follows that $E \times_B E$ can be embedded into a finite-dimensional Euclidean space.
    
    Since $E$ is a subspace of a Euclidean space, it follows that $E$ is separable.
    Hence, by \Cref{prop: ENR implies ANR} and \Cref{prop: fibre product and path space are ANR}, it follows that $E \times_B E$ is a separable ANR.
    Note that isotropy groups of $E \times_B E$ under the transposition action of $\Z_2$ are the trivial group $T$ and $\Z_2$ itself. 
    As the fixed point sets
    $$
        (E \times_B E)^T = E \times_B E
            \text{ and }
        (E \times_B E)^{\Z/2} = \Delta E \cong E
    $$ 
    are ANRs, it follows from \cite[Theorem 2.1]{Jaworowski1976} that $E \times_B E$ is a $\Z_2$-ENR. 
    Note that $E$ equipped with trivial $\Z_2$-action is a $\Z_2$-ANR, see \cite[Corollary 8.2]{murayama-G-ANR}.
    Hence, the desired result follows from \Cref{thm: cobfib}.
\end{proof}

\subsection{Equivariant Sectional Category}
\label{subsec: equi-sectional-category}
\hfill\\ \vspace{-0.7em}

Schwarz \cite{Sva} introduced and studied the notion of the sectional category of a fibration, which was later extended by Bernstein and Ganea \cite{secat} to arbitrary maps. 
The equivariant analogue was introduced by Colman and Grant \cite{EqTC}.
We recall its definition together with an equivalent criterion expressed via a global section of fibred joins of maps.

\begin{definition}
The \emph{equivariant sectional category} of a map $p \colon E \to B$, denoted by $\sct_G(p)$, is the least nonnegative integer $k$ for which $B$ admits a $G$-invariant open cover $\{U_0,\dots,U_k\}$ such that on each $U_i$ there exists a $G$-map $s_i \colon U_i \to E$ for which $p \circ s_i$ is $G$-homotopic to the inclusion $\iota_{U_i} \colon U_i \hookrightarrow B$.
If such an integer does not exist, we set $\sct_G(p) = \infty$.
\end{definition}

Let $F$ be a topological space, and let $k$ be a non-negative integer.
The \emph{$(k+1)$-fold fibred join} of $F$ is defined as the space
$$
    J^{k}(F) 
        := \left\{(f_0,t_0,\dots,f_k,t_k) \,\big\vert\,
                f_i \in F, \, t_i \in I, \, \sum t_i = 1 \right\}/\sim\,,
$$
where the equivalence relation $\sim$ is generated by
$$
    (f_0,t_0,\dots,f_i,0, \dots,f_k,t_k) \sim (f_0,t_0,\dots,f_i',0,\dots,f_k,t_k)\,.
$$
This construction generalizes to the fibrewise setting as follows.
Let $p \colon E \to B$ be a $G$-map. 
The \emph{$(k+1)$-fold fibred join} of $E$ is defined as the space
$$
    J^{k}_B(E) 
        := \left\{(e_0,t_0,\dots,e_k,t_k) \,\big\vert\,
                e_i \in E, \, t_i \in I, \, p(e_i) = p(e_j), \, \sum t_i = 1 \right\}/\sim\,,
$$
where the equivalence relation $\sim$ is defined in the same way as in the definition of $J^k(F)$.
The space $J^{k}_B(E)$ is equipped with a $G$-action given by
$$
    g \cdot [e_0,t_0,\dots,e_k,t_k] = [g \cdot e_0,t_0,\dots,g \cdot e_k,t_k]
$$
The \emph{$(k+1)$-fold fibred join} of $p$ is then defined as the $G$-map $p_k \colon J^{k}_{B}(E) \to B$ given by
$$
    p_k([e_0,t_0,\dots,e_k,t_k]) = p(e_0) = \dots = p(e_k)\,.
$$
We note that the notation $J^k$ denotes the $(k+1)$-fold join of the object under consideration.

\begin{lemma}[{\cite[Lemma 3.3]{Grantsymmtc}}]
If $p \colon E \to B$ is a $G$-fibration with fibre $F$, then $p_k \colon J^{k}_B(E) \to B$ is a $G$-fibration with fibre $J^{k}(F)$.   
\end{lemma}

\begin{proposition}[{\cite[Proposition 3.4]{Grantsymmtc}}]
Suppose $p \colon E \to B$ is a $G$-fibration over a paracompact base.
Then $\sct_G(p) \leq k$ if and only if $p_k \colon J^k_B(E) \to B$ admits a global $G$-section.
\end{proposition}

\subsection{Relative Category}
\label{subsec: relative category}
\hfill\\ \vspace{-0.7em}

Suppose $p \colon E \to B$ is a map.
Then we have the following commutative diagram
\begin{equation}
\label{diag: canonical map into join}
\begin{tikzcd}
E \arrow[rr, "i_k"] \arrow[rd, "p"'] &        & J^k_B(E) \arrow[ld, "p_k"] \\
                                     & B, &                           
\end{tikzcd}
\end{equation}
where $i_k \colon E \to J^{k}_B(E)$ is the canonical map defined as
$$
    i_k(e)
        = \left[e,\frac{1}{k+1},\dots,e,\frac{1}{k+1}\right].
$$

\begin{definition}[{\cite[Definition 3]{secat-and-relcat-I}}]
\label{defn: relative category}
Suppose $p \colon E \to B$ is a map. 
The \emph{relative category} of $p$, denoted by $\rct(p)$, is the least nonnegative integer $k$ such that $p_k \colon J^k_B(E) \to B$ admits a homotopy section $\sigma$ satisfying $\sigma \circ p \simeq i_k$.
If such an integer does not exist, we set $\rct(p) = \infty$.
\end{definition}

\begin{proposition}[{\cite[Corollary 1.7]{garcia2019note}}]
\label{prop: equivalent definition of relative category}
    Let $i_X \colon A \to X$ be a closed cofibration, where $X$ is a normal space. 
    Then the following statements are equivalent:
    \begin{enumerate}
        \item $\mathrm{relcat}(i_X) \leq k$.
        \item $X$ admits an open cover $U_0, U_1, \dots, U_k$ such that for all $i$, $A \subset U_i$ and there exists a homotopy $H_i \colon U_i \times I \to X$ satisfying $H_i(x,0)=x$, $H_i(x,1) \in A$ and $H_i(a,t)=a$, for all $x \in X$, $a \in A$ and $t \in I$.
        \item $X$ admits an open cover $U_0, U_1, \dots, U_k$ such that for all $i$, $A \subset U_i$ and there exists a homotopy $H_i \colon U_i \times I \to X$ satisfying $H_i(x,0)=x$, $H_i(x,1) \in A$ and $H_i(a,1)=a$, for all $x \in X$ and $a \in A$.
    \end{enumerate}
\end{proposition}

\begin{definition}[{\cite[Definition 2.11]{garcia2019note}}]
Let $i_X \colon A \hookrightarrow X$ be a closed cofibration.
\begin{enumerate}
    \item A subset $U$ of $X$ is said to be \emph{relatively sectional} if $A \subseteq U$, and there exists a homotopy of pairs $H \colon (U,A) \times I \to (X,A)$ satisfying $H(x,0) = x$ and $H(x,1) \in A$ for all $x \in U$.

    
    \item The \emph{generalized relative category} of $i_X$, denoted by $\rct_g(i_X)$, is defined as the least nonnegative integer $k$ such that $X$ admits a (not necessarily open) cover by $k+1$ relatively sectional subsets.
    If such an integer does not exist, we set $\rct_g(i_X) = \infty$.
\end{enumerate}
\end{definition}

\begin{theorem}[{\cite[Theorem 2.16]{garcia2019note}}]
\label{thm: relcat = gen-relcat}
Let $i_X \colon A \hookrightarrow X$ be a closed cofibration between ANR spaces. 
Then $\mathrm{relcat}_g(i_X) = \mathrm{relcat}(i_X)$.
\end{theorem}

\begin{definition}[{\cite[Definition 3.1]{aguilar-gonzalez-2023motion}}]
Let $i_X \colon A \hookrightarrow X$ be a closed cofibration.
The \emph{Fadell-Husseini relative category} of $i_X$, denoted by $\rct^{FH}_{\op}(i_X)$, is defined as the least nonnegative integer $k$ for which $X$ admits an open cover $\{U_0, \dots,U_k\}$ such that 
\begin{enumerate}
    \item $U_0$ is a relatively sectional open subset, 
    \item for $i \geq 1$, $U_i \cap A = \emptyset$ and there are homotopies $H_i \colon U_i \times I \to X \times I$ such that $H_i(x,0) = x$ and $H_i(x,1) \in A$ for all $x \in U_i$.
\end{enumerate}
If such an integer does not exist, we set $\rct^{FH}_{\op}(i_X) = \infty$.
The generalized notion of the Fadell–Husseini relative category of $i_X$, denoted by $\rct_g^{FH}(i_X)$, is defined in the same way, except that the covering sets are not required to be open.
\end{definition}

\begin{theorem}[{\cite[Proposition 3.3]{aguilar-gonzalez-2023motion}}]
\label{thm: relcat = FH-relcat}
Let $i_X \colon A \hookrightarrow X$ be a closed cofibration, where $X$ is a normal space. 
Then $\rct(i_X) = \rct_{\op}^{FH}(i_X)$.
\end{theorem}

\begin{theorem}[{\cite[Proposition 3.7]{aguilar-gonzalez-2023motion}}]
\label{thm: gen-relcat = gen-FH-relcat}
Let $i_X \colon A \hookrightarrow X$ be a closed cofibration between ANR spaces. 
Then $\mathrm{relcat}_g(i_X) = \mathrm{relcat}_g^{FH}(i_X)$.
\end{theorem}

\subsection{Monoidal topological complexity}
\label{subsec: mono-tc}
\hfill\\ \vspace{-0.7em}

The monoidal topological complexity of a topological space was introduced by Iwase and Sakai in \cite{I-S}. 
The motivation behind defining this notion was that the motion of local motion planners should be static when the initial and final states of the mechanical system coincide.

Suppose $\pi \colon X^I \to X \times X$ is the free path space fibration.

\begin{definition}[{\cite[Definition 1.3]{I-S}}]
The \emph{monoidal topological complexity} of a space $X$, denoted by $\TC^{M}(X)$, is the least nonnegative integer $k$ such that $X \times X$ may be covered by open subsets $U_0,\dots,U_k$, each of which contains the diagonal $\Delta X$ and admits a local section $\sigma_i \colon U_i \to X^I$ of $\pi$ such that $\sigma_i(x,x) = \mathrm{c}_x$ for all $x \in X$.
If such an integer does not exist, we set $\TC^{M}(X) = \infty$.
\end{definition}

\begin{definition}
The \emph{D-monoidal topological complexity} of a space $X$, denoted by $\TC^{DM}(X)$, is the least nonnegative integer $k$ such that $X \times X$ may be covered by open subsets $U_0,\dots,U_k$, each of which admits a local section $\sigma_i \colon U_i \to X^I$ of $\pi$ such that $\sigma_i(x,x) = \mathrm{c}_x$ for all $x \in X$ with $(x,x) \in U_i$. 
If such an integer does not exist, we set $\TC^{DM}(X) = \infty$.
\end{definition}

It is clear that $\TC(X) \leq \TC^{DM}(X) \leq \TC^{M}(X)$. 
On the other hand, if $X$ is an ENR, then Dranishnikov \cite{dranishnikov2014topological} noticed that the condition $\Delta X \subseteq U_i$ in Iwase-Sakai's definition can be relaxed, i.e., $\TC^{M}(X) = \TC^{DM}(X)$. 
Further, Aguilar-Guzm\'an and J. Gonz\'alez \cite{aguilar-gonzalez-2023motion}  have shown that this observation holds for ANRs as well.

\begin{theorem}[{\cite[Proposition 2.3]{aguilar-gonzalez-2023motion}}]
If $X$ is an ANR, then
$$
    \TC^{M}(X) = \TC^{DM}(X).
$$   
\end{theorem}

Iwase and Sakai in \cite[Theorem 1.13]{I-S} proved that $\TC^{M}(X) = \TC(X)$, if $X$ is a locally finite simplicial complex.
But it was later withdrawn after Garcia-Calcines pointed out that there is a problem in the proof.
Later, Dranishnikov \cite{dranishnikov2014topological} showed that the statement holds when $X$ is a Lie group, as well as when $\dim(X) < (k+1)(\TC(X)+1)-1$, where $k$ is the connectivity of $X$.

\begin{conjecture}[Iwase-Sakai Conjecture] 
If $X$ is a locally finite simplicial complex, then 
$$
    \TC^{M}(X) = \TC(X).
$$
\end{conjecture}

Recall that a space $X$ is said to be \emph{locally equiconnected}, or \emph{LEC}, if the diagonal map $\Delta \colon X \to X \times X$ is a closed fibration.
Examples of LEC spaces include CW complexes and ANRs.

\begin{theorem}[{\cite[Theorem 12]{carrasquel2014relative}, \cite[Proposition 2.17]{garcia2019note}}]
\label{thm: monoidal tc is a relative category}
If $X$ is a LEC space, then
$$
    \TC^M(X) = \rct(\Delta_X \colon X \to X\times X).
$$
In particular, $\TC^{M}(X)$ is a homotopy invariant when we restrict to LEC spaces.
\end{theorem}

We now show that, in Dranishnikov's result \cite[Theorem 2.5]{dranishnikov2014topological}, the dimension of the space can be replaced by its homotopy dimension, thereby strengthening the result.

\begin{theorem}\label{thm: dim--hdim in Dranishnikov's result}
Suppose $X$ is a $k$-connected CW complex.
If $\hdim(X) < (k+1)(\TC(X)+1)-1$, then
$$
    \TC^{M}(X) = \TC(X).
$$
\end{theorem}

\begin{proof}
Let $Y$ be a CW complex homotopy equivalent to $X$.
Then it follows that $\TC(Y) = \TC(X)$ and $\TC^{M}(Y) = \TC^{M}(X)$.
Hence, it is enough to show that if $\dim(Y) < (k+1)(\TC(Y)+1)-1$, then $\TC^{M}(Y) = \TC(Y)$.

Note that the homotopy fibre of $\Delta_Y$ is the based loop space $\Omega Y$, since the free path space fibration $\pi \colon PY \to Y \times Y$ is the fibrational substitute of $\Delta_Y$.
Since $\Omega Y$ is $(k-1)$-connected, it follows that the map $\Delta_Y$ is a $k$-equivalence (i.e., the induced map $(\Delta_Y)_* \colon \pi_i(Y) \to \pi_i(Y \times Y)$ is an isomorphism for $i < k$ and a surjection for $i = k$).
Then, by \cite[Theorem 14]{secat-and-relcat-II}, it follows that 
$$
    \rct(\Delta_Y) = \sct(\Delta_Y) = \TC(Y)\,.
$$
Hence, the result follows from \Cref{thm: monoidal tc is a relative category}.
\end{proof}

As an application of the preceding theorem, we determine the monoidal topological complexity of configuration spaces.

\begin{example}
\label{ex: monoidal TC of conf spaces}
The \emph{ordered configuration space} of $n$ points in $\R^d$, denoted by $F(\R^d,n)$, is defined as
\[
    F(\R^d,n):=\{(x_1,\dots,x_n)\in (\R^d)^n \mid x_i\neq x_j ~\text{ for }~ i\neq j\}.
\]
For $d,n \geq 2$, Farber and Grant \cite{TC-of-configuration-spaces} showed that 
$$
    \TC(F(\R^d,n)) = 
        \begin{cases}
            2n-2, & \text{if } d \text{ is odd},\\
            2n-3, & \text{if } d \text{ is even}.
        \end{cases}
$$
Since the space $F(\R^d,n)$ is $(d-2)$-connected,
and its homotopy dimension is $(d-1)(n-1)$, the preceding theorem implies that
$$
    \TC^M(F(\R^d,n)) = 
        \begin{cases}
            2n-2, & \text{if } d \text{ is odd},\\
            2n-3, & \text{if } d \text{ is even and } (d,n) \neq (2,2).
        \end{cases}
$$
For the case $d=n=2$, the space $F(\R^2,2)$ is homeomorphic to $\R^3 \times S^1$.
Hence, by \Cref{thm: monoidal tc is a relative category} and \cite[Lemma 2.7]{dranishnikov2014topological}, it follows that
$$ 
    \TC^M(F(\R^2,2)) 
        = \TC^M(\R^3 \times S^1) 
        = \TC^M(S^1)
        = \TC(S^1)
        = 1
$$
Hence, the monoidal topological complexity of configuration spaces is given by
$$
    \TC^M(F(\R^d,n)) = 
        \begin{cases}
            2n-2, & \text{if } d \text{ is odd},\\
            2n-3, & \text{if } d \text{ is even}.
        \end{cases}
$$
for all $n,d \geq 2$.
\end{example}

\begin{definition}[{\cite[Definition 1.2]{aguilar-gonzalez-2023motion}}]
The \emph{Fadell-Husseini monoidal topological complexity} of a space $X$, denoted by $\TC^{FH}(X)$, is the least nonnegative integer $k$ such that $X \times X$ may be covered by open subsets $U_0,\dots,U_k$, each of which admits a local section $\sigma_i \colon U_i \to X^I$ of $\pi$ such that 
\begin{enumerate}
    \item the diagonal $\Delta X$ is a subset of $U_0$,
    \item $\sigma_0(x,x) = \mathrm{c}_x$ for all $x \in X$, 
    \item $\Delta X \cap U_i = \emptyset$ for all $i \geq 1$.
\end{enumerate}
If such an integer does not exist, we set $\TC^{FH}(X) = \infty$.
\end{definition}

\begin{theorem}[{\cite[Theorem 1.3]{aguilar-gonzalez-2023motion}}]
If $X$ is an ANR, then
$$
    \TC^{FH}(X) = \TC^{M}(X).
$$      
\end{theorem}

\subsection{Symmetrized topological complexity}
\label{subsec: symm-tc}
\hfill\\ \vspace{-0.7em}

In this section, we recall symmetrized topological complexity. 
Suppose $X$ is a topological space.
Then there is a $\Z_2 = \langle g \rangle$ action on both $X^I$ and $X \times X$ defined as follows: 
\[
    g\cdot \gamma(t)=\gamma(1-t) ~\text{and}~ g\cdot (x_1,x_2)=(x_2,x_1).
\]
for $\gamma \in X^I$ and $(x_1,x_2)\in X\times X$.
In fact, Grant proved that the free path space fibration $\pi\colon X^I\to X\times X$ is a $\Z_2$-fibration under this action. 

\begin{notation*}
Throughout this subsection, we denote the group $\Z_2$ by $G$.    
\end{notation*}

\begin{definition}[{\cite[Definition 4.1]{gonzalezhighertc}}]
The \emph{symmetrized topological complexity} of a space $X$ is defined as
$$
    \TC^{\Sigma}(X): = \sct_G(\pi).
$$
\end{definition}

\begin{definition}[{\cite[Definition 5.1]{Grantsymmtc}}]
The \emph{monoidal symmetrized topological complexity} of a space $X$, denoted by $\TC^{M,\,\Sigma}(X)$, is the least nonnegative integer $k$ such that $X \times X$ may be covered by $G$-invariant open subsets $U_0,\dots,U_k$, each of which contains the diagonal $\Delta X$ and admits a local $G$-section $\sigma_i \colon U_i \to X^I$ of $\pi$ such that $\sigma_i(x,x) = \mathrm{c}_x$ for all $x \in X$.
If such an integer does not exist, we set $\TC^{M,\Sigma}(X) = \infty$.
\end{definition}

\begin{definition}
The \emph{D-monoidal symmetrized topological complexity} of a space $X$, denoted by $\TC^{DM,\,\Sigma}(X)$, is the least nonnegative integer $k$ such that $X \times X$ may be covered by $G$-invariant open subsets $U_0,\dots,U_k$, each of which admits a local $G$-section $\sigma_i \colon U_i \to X^I$ of $\pi$ such that $\sigma_i(x,x) = \mathrm{c}_x$ for all $x \in X$ with $(x,x) \in U_i$.
If such an integer does not exist, we set $\TC^{DM,\Sigma}(X) = \infty$.
\end{definition}

It is clear that 
    $$
        \TC^{\Sigma}(X)
            \leq \TC^{DM,\,\Sigma}(X)
            \leq \TC^{M,\,\Sigma}(X).
    $$
    
\begin{theorem}[{\cite[Theorem 5.2]{Grantsymmtc}}]
Let $X$ be a paracompact ENR. 
Then
$$
    \TC^{\Sigma}(X)
        = \TC^{DM,\,\Sigma}(X)
        = \TC^{M,\,\Sigma}(X).
$$
\end{theorem}

Since symmetrized topological complexity is a homotopy invariant \cite[Proposition 4.7]{gonzalezhighertc}, it follows that Grant's dimension-connectivity upper bound \cite[Theorem 4.2]{Grantsymmtc} can be strengthened as follows.
We leave the details to the reader.

\begin{theorem}
If $X$ is a $s$-connected CW complex, then
$$
    \TC^{\Sigma}(X) < \frac{2\,\hdim(X) + 1}{s+1}.
$$
\end{theorem}

As an application of the preceding theorem, we obtain the following estimate for the monoidal symmetrized topological complexity of configuration spaces.

\begin{example}
Consider the configuration space $F(\R^d,n)$ for $d,n \geq 2$.
Then 
$$
    \TC^{\Sigma}(F(\R^d,n)) 
        < \frac{2 (d-1)(n-1)+1}{(d-2)+1} 
        = 2(n-1) + \frac{1}{(d-1)}.
$$
Hence, $\TC^{\Sigma}(F(\R^d,n)) \leq 2n-2$.
Since $\TC(F(\R^d,n)) \leq \TC^{\Sigma}(F(\R^d,n))$, it follows by \cite[Theorem 1]{TC-of-configuration-spaces} that
$$
    \TC^{\Sigma}(F(\R^d,n)) = 
        \begin{cases}
            2n-2, & \text{if } d \text{ is odd},\\
            2n-2\text{ or }2n-3, & \text{if } d \text{ is even}.
        \end{cases}
$$
Moreover, $\TC^{M,\Sigma}(F(\R^d,n)) = \TC^{\Sigma}(F(\R^d,n))$.
If $n=2$, then $F(\R^d,2)$ is homotopy equivalent to $S^{d-1}$.
Hence, by \cite[Theorem 6.1]{Grantsymmtc} (see also \cite{MR3665579}), it follows that
$$
    \TC^{\Sigma}(F(\R^d,2))
        = \TC^{\Sigma}(S^{d-1})
        = 2.
$$
Hence, 
$$
    \TC^{M,\Sigma}(F(\R^d,n)) = 
        \begin{cases}
            2 & \text{if } n=2,\\
            2n-2, & \text{if } d \text{ is odd},\\
            2n-2\text{ or }2n-3, & \text{if } n \geq 3 \text{ and } d \text{ is even}.
        \end{cases}
$$
\end{example}

\subsection{Parametrized topological complexity}
\label{subsec: para-tc}
\hfill\\ \vspace{-0.7em}

For a fibration $p \colon E\to B$, consider the subspace of the path space $E^I$, defined as follows:
\[
    E^I_B := \{\gamma\in E^I \mid (p\circ \gamma)(t)=b ~\text{for some}~ b\in B ~\text{and for all}~ t\in I \}.
\]
The fibre product corresponding to $p \colon E\to B$ is defined by 
\[
    E\times_B E:=\{(e_1,e_2)\in E\times E \mid p(e_1)=p(e_2)\}.
\]
Then the map
\begin{equation}
    \Pi \colon E^I_B \to E\times_B E, \quad \Pi(\gamma) = (\gamma(0),\gamma(1))
\end{equation}
is a fibration (see \cite[Appendix]{PTCcolfree}).

\begin{definition}[{\cite[Definition 4.1]{farber-para-tc}}]
The \emph{parametrized topological complexity} of a fibration $p \colon E\to B$, denoted by $\TC[p \colon E\to B]$, is defined as 
\[
    \TC[p \colon E\to B] :=\sct(\Pi).
\]
\end{definition}

Cohen, Farber, and Weinberger \cite{farber-para-tc} showed that $\TC[p \colon E\to B]$ is a fibre-preserving homotopy invariant and also proved that 
$$
    \TC(F) \leq \TC[p \colon E \to B],
$$ 
where $F$ denotes the fibre of the fibration $p \colon E \to B$.

\begin{definition}
\label{defn: gen-para-tc}
The \emph{generalized parametrized topological complexity} of a fibration $p \colon E\to B$, denoted by $\TC_g[p \colon E\to B]$, is the least nonnegative integer $k$ such that $E \times_B E$ may be covered by (not necessarily open) subsets $S_0, \dots, S_k$, each of which admits a local section $\sigma_i \colon S_i \to E^I_B$ of $\Pi$.
If such an integer does not exist, we set $\TC_{g}[p \colon E \to B] = \infty$.
\end{definition}

\begin{theorem}
\label{thm:  equivalent defn of para tc}
Suppose $p \colon E \to B$ is a fibration.
Let $\Delta_E \colon E \to E \times_B E$ denote the diagonal map.
If $E \times_B E$ is paracompact, then the following statements are equivalent:
\begin{enumerate}
    \item $\TC[p \colon E \to B] \leq k$.
    \item $\sct(\Delta_E) \leq k$.
    \item The fibration $\Pi_k \colon J^k_{E \times_B E}(E^I_B) \to E \times_B E$ admits a global section.
\end{enumerate}
\end{theorem}

\begin{proof}
Since $E \times_B E$ is paracompact, it follows from \cite[Theorem 3]{Sva} that (1) and (3) are equivalent.
Let $h \colon E \to E^I_B$ be the homotopy equivalence which maps $e \in E$ to the constant path in $E$ which takes the value $e$.
Since the following diagram commutes
\begin{equation}
\label{diag: E is homotopy equivalent to E^I_B}
\begin{tikzcd}
E \arrow[rd, "\Delta_E"'] \arrow[rr, "h"] &                   & E^{I}_{B} \arrow[ld, "\Pi"] \\
                                          & {E \times_B E\,,} &                            
\end{tikzcd}
\end{equation}
it follows that $\sct(\Delta_E) = \sct(\Pi) = \TC[p \colon E \to B]$.
\end{proof}

\section{Monoidal parametrized topological complexity}
\label{sec: mono-para-tc}

In this section, we introduce and study the notion of monoidal parametrized topological complexity. The motivation for defining this notion is that the motion of parametrized local motion planners should be static when the initial and final states of the mechanical system coincide.
\begin{definition}
The \emph{monoidal parametrized topological complexity} of a fibration $p \colon E \to B$, denoted by $\TC^{M}[p \colon E\to B]$, is the least nonnegative integer $k$ such that $E \times_B E$ may be covered by open subsets $U_0,\dots,U_k$, each of which contains the diagonal $\Delta E$ and admits a local section $\sigma_i \colon U_i \to E^I_B$ of $\Pi$ such that $\sigma_i(e,e) = \mathrm{c}_e$ for all $e \in E$.
If such an integer does not exist, we set $\TC^{M}[p \colon E \to B] = \infty$.
\end{definition}

\subsection{Properties}

Note that we always have
$$
     \TC(F) 
        \leq \TC[p \colon E \to B] 
        \leq \TC^{M}[p \colon E \to B]\,,
$$
where $F$ denotes the fibre of the fibration $p \colon E \to B$.
Moreover, $\TC^M(F)$ is not well defined, since monoidal topological complexity is not a homotopy invariant.
If $B=\{*\}$, then $\TC^{M}[p \colon E \to B]=\TC^{M}(E)$. 
Hence, monoidal parametrized topological complexity is not a fibre-homotopy invariant, because monoidal topological complexity itself is not homotopy invariant.
However, we have the following result:

\begin{proposition}
\label{prop: mono-para-tc under pullback}
Let $p \colon E \to B$ be a fibration and $B' \subseteq B$. 
If $E' = p^{-1}(B')$ and $p' \colon E' \to B'$ is the restricted fibration over $B'$, then 
\begin{equation}\label{eq: restricted fibration ineq for monoidal parametrized}
\TC^{M}[p' \colon E' \to B'] \leq \TC^{M}[p \colon E \to B]\,. 
\end{equation}
In particular, 
\begin{enumerate}
\item $\TC^{M}(p^{-1}(b)) \leq \TC^{M}[p \colon E \to B]$ for all $b \in B$.

\item If $p \colon E \to B$ is a locally trivial fibration with fibre $F$, then $\TC^{M}(F) \leq \TC^{M}[p \colon E \to B]$.
\end{enumerate}
\end{proposition}

\begin{proof}
Note that the following commutative diagram   
\[
\begin{tikzcd}
(E')^{I}_{B'} \arrow[d, "\Pi'"'] \arrow[r, hook] & E^{I}_B \arrow[d, "\Pi"] \\
E' \times_{B'} E' \arrow[r, hook]                & E \times_B E            
\end{tikzcd}
\]
is a pullback.
Suppose $U$ is an open subset of $E \times_B E$ containing the diagonal $\Delta E$ and $\sigma \colon U \to E^{I}_B$ is a section of $\Pi$ satisfying $\sigma(e,e) = \mathrm{c}_e$.
Then $V := U \cap (E' \times_{B'} E')$ contains $\Delta E'$.
Moreover, by the universal property of pullbacks, there exists a section $\sigma' \colon V \to (E')^{I}_{B'}$ of $\Pi'$ such that the following diagram commutes
\[
\begin{tikzcd}
V \arrow[d, hook] \arrow[r, "\sigma'"] & (E')^{I}_{B'} \arrow[d, hook] \\
U \arrow[r, "\sigma"]                  & {E^{I}_B\,.}                 
\end{tikzcd}
\]
Hence, $\sigma'(e',e') = \mathrm{c}_{e'}$ for all $e' \in E'$.
This proves the inequality \eqref{eq: restricted fibration ineq for monoidal parametrized}.

In particular, taking $B'=\{b\}$ in \eqref{eq: restricted fibration ineq for monoidal parametrized}, we obtain $\TC^{M}(p^{-1}(b)) \leq \TC^{M}[p \colon E \to B]$. 
Furthermore, for a locally trivial fibration $p\colon E\to B$ with fibre $F$, we have $F$ is homeomorphic to $p^{-1}(b)$ for any $b\in B$. 
Consequently, we have $\TC^{M}(F) \leq \TC^{M}[p \colon E \to B]$. 
\end{proof}

Let $s_0 \colon E \to E^I_B$ be the canonical map, which sends an element $e$ in $E$ to the constant path $\mathrm{c}_e$ that takes the value $e$.
It induces a map $s_k \colon E \to J^k_{E \times_B E}(E^{I}_B)$ given by
$$
    s_k(e) = \left[\mathrm{c}_e,\frac{1}{k+1},\dots,\mathrm{c}_e,\frac{1}{k+1}\right].
$$
Let $\Delta_E \colon E \to E \times_B E$ be the diagonal map.
The diagram \eqref{diag: E is homotopy equivalent to E^I_B} then induces the following commutative diagram
\begin{equation}
\begin{tikzcd}
E \arrow[rr, "i_k"] \arrow[rrd, "s_k"] \arrow[d, "h"'] &  & J_{E \times_B E}^{k}(E) \arrow[r, "(\Delta_E)_k"] \arrow[d, "h_k"] & E \times_B E \arrow[d, Rightarrow, no head] \\
E^{I}_B \arrow[rr, "j_k"]                              &  & J_{E \times_B E}^{k}(E^{I}_B) \arrow[r, "\Pi_k"]                   & E \times_B E ,                              
\end{tikzcd}   
\end{equation}
where $i_k$ and $j_k$ are the canonical maps satisfying $\Pi_k \circ j_k = \Pi$ and $(\Delta_E)_k \circ i_k = \Delta_E$, respectively, as in diagram \eqref{diag: canonical map into join}, and $h_k$ is a homotopy equivalence given by
$$
   h_k(\left[e_0,t_0,\dots,e_k,t_k\right])
        = \left[\mathrm{c}_{e_0},t_0,\dots,\mathrm{c}_{e_k},t_k\right].
$$

The following theorem states that the monoidal parametrized topological complexity is a relative category. 
Moreover, it generalizes \cite[Theorem 12]{carrasquel2014relative} and \cite[Proposition 2.4 (2)]{dranishnikov2014topological} to the parametrized setting.

\begin{theorem}
\label{thm: equivalent defn of monoidal para tc}
Suppose $p \colon E \to B$ is an LEC fibration such that $E \times_B E$ is paracompact.
Then the following statements are equivalent:
\begin{enumerate}
    \item $\TC^{M}[p \colon E\to B] \leq k$.
    
    \item $\mathrm{relcat}(\Delta_E) \leq k$.

    \item The fibration $\Pi_k \colon J^k_{E \times_B E}(E^I_B) \to E \times_B E$ admits a global section $\sigma$ satisfying $\sigma \circ \Delta_E \simeq s_k$.
    
    \item The fibration $\Pi_k \colon J^k_{E \times_B E}(E^I_B) \to E \times_B E$ admits a global section $\sigma$ satisfying $\sigma \circ \Delta_E = s_k$.
\end{enumerate}
\end{theorem}

\begin{proof}
$(1) \implies (3)$
    Suppose $\TC^{M}[p \colon E \to B] \leq k$.
    Let $\mathcal{U} = \{U_0, \dots,U_k\}$ be an open cover of $E \times_B E$ such that $\Delta E \subset U_i$ for each $i$, and there exists a section $\sigma_i$ of $\Pi$ over each $U_i$ satisfying $\sigma_i(e,e) = \mathrm{c}_e$ for all $e \in E$.
    As $E \times_B E$ is paracompact, there exists a partition of unity $\{h_0,\dots,h_k\}$ subordinate to the cover $\mathcal{U}$.
    Then 
    $$
        \sigma(e_1,e_2) 
            = [\sigma_0(e_1,e_2),h_0(e_1,e_2), \dots \sigma_k(e_1,e_2),h_k(e_1,e_2)]
    $$
    is a global section of $\Pi_k$.
    Let $K \colon E \times I \to J^{k}_{E \times_B E}(E^{I}_B)$ be the homotopy defined as
    $$
        K(e,t) 
            = \left[\mathrm{c}_e,(1-t) h_{0}(e,e)+t\frac{1}{k+1}, \dots, \mathrm{c}_e,(1-t) h_{k}(e,e)+t\frac{1}{k+1}\right].
    $$
    Then $K_0 = \sigma \circ \Delta_E$ and $K_1 = s_k$ implies that $\Pi_k$ admits a global section $\sigma$ satisfying $\sigma \circ \Delta_E \simeq s_k$.

$(3) \implies (2)$ 
    Suppose $\Pi_k$ admits a global section $\sigma$ satisfying $\sigma \circ \Delta_E \simeq s_k$.
    If $g_k \colon J^{k}_{E \times_B E}(E^{I}_B) \to J^{k}_{E \times_B E}(E^{I}_B)$ is a homotopy inverse of $h_k$, then
    $$
        g_k \circ \sigma \circ \Delta_E
            \simeq g_k \circ s_k
            = g_k \circ h_k \circ i_k
            \simeq i_k
    $$
    implies $g_k \circ \sigma$ is a global homotopy section of $(\Delta_E)_k$ satisfying $(g_k \circ \sigma) \circ \Delta_E \simeq i_k$.
    Hence, $\mathrm{relcat}(\Delta_E) \leq k$.
    
$(2) \implies (3)$
    Suppose $\mathrm{relcat}(\Delta_E) \leq k$, that is, $(\Delta_E)_k$ admits a global homotopy section $\rho$ satisfying $\rho \circ \Delta_E \simeq i_k$.
    Then $\rho' = h_k \circ \rho$ is a global homotopy section of $\Pi_k$ satisfying 
    $$
        \rho' \circ \Delta_E 
            = h_k \circ \rho \circ \Delta_E
            \simeq h_k \circ i_k
            = s_k.
    $$
    Since $\Pi_k$ is a fibration, we can get a global section $\sigma$ of $\Pi_k$ such that $\sigma \simeq \rho'$.
    Hence, $\sigma \circ \Delta_E \simeq \rho' \circ \Delta_E \simeq s_k$.

$(3) \implies (4)$
Suppose $\Pi_k$ admits a global section $\sigma$ satisfying $\sigma \circ \Delta_E \simeq s_k$.
Let $K \colon \Delta E \times I \to J^{k}_{E \times_B E}(E^{I}_B)$ be a homotopy such that $K_0 = \left.\sigma\right|_{\Delta E}$ and $K_1 = s_k \circ \Delta_E^{-1}$, where $\Delta_E^{-1} \colon \Delta E \to E$ is the inverse of the restricted diagonal map.
Consider the diagram
\begin{equation}
\label{diag: relative homotopy lifting}
\begin{tikzcd}
(E \times_B E \times \{0\}) \cup (\Delta E \times I) \arrow[d, hook] \arrow[rr, "\sigma \cup K"] &  & J^{k}_{E \times_B E}(E^I_B) \arrow[d, "\Pi_k"] \\
E \times_B E \times I \arrow[rr, "H"] \arrow[rru, "F", dotted]                                   &  & E \times_B E,                                  
\end{tikzcd}
\end{equation}
where $H$ is the identity homotopy.
As the inclusion map $\Delta E \hookrightarrow E \times_B E$ is a closed cofibration, by \cite[Proposition 2.9]{Grantsymmtc}, we obtain a homotopy $F \colon (E\times_B E) \times I \to J^{k}_{E \times_B E}(E^I_B)$ such that the diagram \eqref{diag: relative homotopy lifting} commutes.
Let $\widetilde{\sigma} \colon E \times_B E \to J^{k}_{E \times_B E}(E^{I}_B)$ be defined as
$$
    \widetilde{\sigma} := \left.F\right|_{(E \times_B E) \times \{1\}}.
$$
Then $\widetilde{\sigma}$ is a section of $\Pi_k$, since 
$$
    \Pi_k \circ \widetilde{\sigma} 
        = \Pi_k \circ \left.F\right|_{(E \times_B E) \times \{1\}} 
        = \left.H\right|_{(E \times_B E) \times \{1\}} 
        = \mathrm{id}_{E \times_B E}.
$$    
Moreover, $\widetilde{\sigma}$ satisfies
$$
    \widetilde{\sigma}(e,e) 
        = \left.F\right|_{(E \times_B E) \times \{1\}}(e,e)
        = K(e,e,1)
        = (s_k \circ \Delta_E^{-1})(e,e)
        = s_k(e).
$$
    
$(4) \implies (1)$
    Suppose $\Pi_k$ admits a global section $\sigma$ satisfying $\sigma \circ \Delta_E = s_k$.
    Let $U_i := \sigma^{-1}(t_i^{-1}(0,1])$, where $t_i$'s are join coordinates in $J^{k}_{E \times_B E}(E^{I}_B)$. 
    As each $U_i$ consists of those $(e_1,e_2) \in E \times_B E$ for which the $i$-th join coordinate of $\sigma(e_1,e_2)$ is positive, it follows that $U_i$'s forms a cover of $E \times_B E$.
    Moreover, each $U_i$ contains the diagonal $\Delta E$, since
    $$
        t_i(\sigma(e,e)) 
            = t_i(s_k(e,e)) 
            = \frac{1}{k+1}.
    $$
    Define $\sigma_i \colon U_i \to E^I_B$ by the formula
    $$
        \sigma(e_1,e_2) 
            = [\sigma_0(e_1,e_2),t_0, \dots ,\sigma_k(e_1,e_2),t_k].
    $$
    As $\sigma(e,e) = s_k(e)$ for all $e \in E$, it follows that $\sigma_i(e,e) = \mathrm{c}_e$.
\end{proof}

We now present some applications of the above theorem.

\begin{proposition}
\label{prop: mon-para-tc = 0 iff para-tc = 0}
Suppose $p \colon E \to B$ is an LEC fibration such that $E \times_B E$ is a paracompact Hausdorff space. 
Then 
$$
    \TC[p \colon E \to B] = 0\text{ if and only if }\TC^{M}[p \colon E \to B] = 0.
$$
In particular, if $p \colon E \to B$ is an LEC fibration with contractible fibre $F$ and $E \times_B E$ is a CW complex, then $\TC^{M}[p \colon E \to B] = 0$.    
\end{proposition}

\begin{proof}
Suppose $\TC[p \colon E \to B] = 0$, i.e., $\Pi$ admits a global section, say $\sigma \colon E \times_B E \to E^{I}_B$.
Let $g \colon E^{I}_B \to E$ be the evaluation map given by $\alpha \mapsto \alpha(0)$.
Let $h \colon E \to E^{I}_B$ be the map which sends $e \in E$ to the constant path $\mathrm{c}_e$ at $e$, see diagram \eqref{diag: E is homotopy equivalent to E^I_B}.
Let $H \colon E^{I}_B \times I \to E^{I}_B$ be the homotopy defined by
$$
    H(\alpha,t)(s) := \alpha(t-st) 
        \quad \text{for all }s,t \in I \text{ and } \alpha \in E^{I}_B\,.
$$
Then $H$ is a homotopy between the identity map of $E^{I}_B$ and $h \circ g$.
In particular, $g$ is a homotopy equivalence, since $g \circ h = \id_E$.
Note that $(h \circ g)(\alpha) = \mathrm{c}_{\alpha(0)}$ for all $\alpha \in E^{I}_B$.
Define the homotopy $K \colon (E \times_B E) \times I \to E \times_B E$ as the composition
$$
    K : = \Pi \circ H \circ (\sigma \times \id_I) \,.
$$
Then $K_0 = \id_{E \times_B E}$ and $K_1 = \Pi \circ H_1 \circ \sigma = \Pi \circ h \circ g \circ \sigma = \Delta_E \circ g \circ \sigma$, where $\Delta_E \colon E \to E \times_B E$ is the diagonal inclusion.
Moreover, 
\begin{align*}
K(e,e,1) 
        = \Pi(H(\sigma(e,e),1)) 
        & = (H(\sigma(e,e),1)(0),\, H(\sigma(e,e),1)(1))\\
        & = (\sigma(e,e)(1),\, \sigma(e,e)(0)) \\
        & = (e,e).     
\end{align*}
Hence, by \Cref{prop: equivalent definition of relative category}, it follows that $\rct(\Delta_E) = 0$.
Consequently, $\TC^{M}[p \colon E \to B] = 0$ by \Cref{thm: equivalent defn of monoidal para tc}.

Furthermore, if the fibre $F$ of $p$ is contractible and $E \times_B E$ is a CW complex, then \cite[Proposition 4.5]{farber-para-tc} implies that $\TC[p \colon E \to B] = 0$. 
Hence, $\TC^{M}[p \colon E \to B] = 0$.
\end{proof}

Let us recall the notion of fibre-preserving homotopy equivalence in the sense of Grant, see \cite[Page 292]{GrantPTC}.
Let $p \colon E \to B$ and $p' \colon E' \to B'$ be fibrations.
A \emph{fibre-preserving map} from $p$ to $p'$ is a pair of maps $f \colon E \to E'$ and $g \colon B \to B'$ satisfying $p' \circ f = g \circ p$, and is denoted by $(f,g) \colon p \to p'$.
A \emph{fibre-preserving homotopy} is a pair of homotopies $F \colon E \times I \to E'$ and $G \colon B \times I \to B'$
such that $p' \circ F = G \circ (p \times \id_I)$, and is denoted by $(F,G)$.
The fibrations $p$ and $p'$ are said to be \emph{fibre-preserving homotopy equivalent} if there exist fibre-preserving maps $(f,g) \colon p \to p'$ and $(f',g') \colon p' \to p$ such that $(f \circ f',\, g \circ g')$ is fibre-preserving homotopic to $(\id_{E'},\, \id_{B'})$ and $(f' \circ f,\, g' \circ g)$ is fibre-preserving homotopic to $(\id_{E},\, \id_{B})$.

\begin{corollary}
Suppose $p \colon E \to B$ and $p' \colon E' \to B'$ are LEC fibrations such that $E \times_B E$ and $E' \times_{B'} E'$ are paracompact. 
If there exist homotopy equivalences $f \colon E \to E'$ and
$g \colon B \to B'$ such that $(f,g) \colon p \to p'$ is a fibre-preserving map, then
$$
    \TC^M[p \colon E \to B] 
        = \TC^M[p' \colon E' \to B'].
$$
\end{corollary}

\begin{proof}
By \cite[Page 53]{JPMay}, it follows that $(f,g)$ is in fact a fibre-preserving homotopy equivalence.
Let $(f',g') \colon p' \to p$ be a fibre-preserving map such that $(f \circ f',\, g \circ g')$ and $(f' \circ f,\, g' \circ g)$ are fibre-preserving homotopic to $(\id_{E'},\, \id_{B'})$ and $(\id_{E},\, \id_{B})$, respectively.
Note that the fibre-preserving maps $p$ and $p'$ induce the following commutative diagram
\[
\begin{tikzcd}
E \arrow[d, "\Delta_E"'] \arrow[r, "f"] & E' \arrow[r, "f'"] \arrow[d, "\Delta_{E'}"] & E \arrow[d, "\Delta_E"] \\
E \times_B E \arrow[r, "f \times f"]    & E' \times_{B'} E' \arrow[r, "f'\times f'"]  & {E \times_{B} E\,.}    
\end{tikzcd}
\]
Since $(f' \circ f,\, g' \circ g)$ is fibre-preserving homotopic to $(\id_{E},\, \id_{B})$, it follows that $(f' \times f') \circ (f \times f) \colon E \times_B E \to E \times_B E$ is homotopic to $\id_{E \times_B E}$.
Hence, by \cite[Corollary 13]{secat-and-relcat-I}, we have $\rct(\Delta_{E'}) \leq \rct(\Delta_E)$.
Similarly, $\rct(\Delta_{E}) \leq \rct(\Delta_{E'})$.
Hence, the desired equality follows from \Cref{thm: equivalent defn of monoidal para tc}.
\end{proof}

For an ENR, the inequalities $\TC(X)\leq \TC^M(X)\leq \TC(X)+1$ were shown in \cite{I-S} and \cite{dranishnikov2014topological}.
The following theorem generalizes these inequalities to the parametrized setting.

\begin{theorem}
Suppose $p \colon E \to B$ is an LEC fibration such that $E \times_B E$ is paracompact.
Then
$$
    \TC[p \colon E\to B] 
        \leq \TC^{M}[p \colon E\to B] 
        \leq \TC[p \colon E\to B] + 1.
$$
\end{theorem}

\begin{proof}
By \Cref{thm:  equivalent defn of para tc} and \Cref{thm: equivalent defn of monoidal para tc}, it is enough to show that
$$
    \sct(\Delta_E) \leq \rct(\Delta_E) \leq \sct(\Delta_E) + 1,
$$
where $\Delta_E \colon E \to E \times_B E$ is the diagonal inclusion.
Hence, the result follows from \cite[Theorem 18]{secat-and-relcat-I} (see also \cite[Proposition 12]{secat-and-relcat-II}).
\end{proof}

An immediate corollary of the above result is as follows:

\begin{corollary}
\label{cor: existence of monoidal section}
Suppose $p \colon E \to B$ is an LEC fibration such that $E \times_B E$ is paracompact.
If $\TC[p \colon E \to B] < \infty$, then there exists an open subset $U$ of $E \times_B E$ containing $\Delta E$ and a section $\sigma \colon U \to E^{I}_B$ such that $\sigma(e,e) = \mathrm{c}_e$ for all $e \in E$.
\end{corollary}

\begin{question}
\label{question: existence of monoidal section}
Is the above corollary true without the assumption $\TC[p \colon E \to B] < \infty$. 

\medskip

Suppose $p \colon E \to B$ is a locally trivial fibration such that $E$ and $B$ are metrizable separable ANRs.
Then $\Pi$ is a fibration between ANR spaces by \cite[Theorem 4.7]{farber-para-tc}.
Let $\widetilde{s}_0 \colon \Delta E \to E^I_B$ be the canonical local section, which sends an element $(e,e)$ in the diagonal $\Delta E$ to the constant path $\mathrm{c}_e$ that takes the value $e$.
Following the proof of \cite[Theorem 2.7]{garcia2019note}, it follows that there exists an open subset $U$ of $E \times_B E$ containing $\Delta E$ and a section $\sigma' \colon U \to E^{I}_B$ of $\Pi$ such that $\left.\sigma'\right|_{\Delta E} \simeq \widetilde{s}_0$. 
If one can show that there exists an open subset $U$ with above properties such that $\Delta E \hookrightarrow U$ is a cofibration, then, using the lifting diagram
\[
\begin{tikzcd}
(U \times \{0\}) \cup (\Delta E \times I) \arrow[d, hook] \arrow[rr, "\sigma' \cup K"] &  & E^I_B \arrow[d, "\Pi"] \\
U \times I \arrow[rr, "H"] \arrow[rru, "F", dotted]                                    &  & {E \times_B E\,,}     
\end{tikzcd}
\]
we can find a section $\sigma := F_1$ satisfying the desired properties. 
Here $K$ is a homotopy between $\left.\sigma'\right|_{\Delta E}$ and $\widetilde{s}_0$, and $H(e_1,e_2,t) = (\sigma'(e_1,e_2)(0),\sigma'(e_1,e_2)(1))$.
\end{question}

\subsection{Parametrized Iwase-Sakai conjecture}
\label{subsec: para-Iwase-Sakai-conj}
\hfill\\ \vspace{-0.7em}

In this section, we present several results establishing cases in which the parametrized Iwase–Sakai conjecture \eqref{conj: para IS conjecture} holds, together with examples illustrating these results. 
Satisfying the parametrized Iwase--Sakai conjecture is crucial because it guarantees that imposing the natural stasis condition (requiring zero motion when the initial and terminal configurations coincide) does not introduce any additional algorithmic complexity, even when the robot's workspace varies with external parameters. 
This provides a theoretical foundation for constructing motion planners that respect the natural stasis condition while retaining the optimal complexity predicted by the underlying topological invariants in parametrized settings.
We begin by extending \cite[Theorem 2.5]{dranishnikov2014topological} to the parametrized setting.

\begin{theorem}
\label{thm: mon-para-tc = para-tc if dim(E) < (TC[p] + 1)(k+1)-1}
Suppose $p \colon E \to B$ is an LEC fibration with the fibre $F$, where $E$ is a CW complex and $E \times_B E$ is paracompact.
If $F$ is $k$-connected such that
$$
    \dim(E) < (\TC[p \colon E \to B]+1) \cdot (k+1) - 1, 
$$
then $\TC^{M}[p \colon E \to B] = \TC[p \colon E \to B]$.
\end{theorem}

\begin{proof}
Let $\Delta_E \colon E \to E \times_B E$ be the diagonal inclusion.
Then the homotopy fibre of $\Delta_E$ is the based loop space $\Omega F$, since $\Pi$ is a fibrational substitute for the map $\Delta_E$, see diagram \eqref{diag: E is homotopy equivalent to E^I_B}.
Since $\Omega F$ is $(k-1)$-connected, it follows that the map $\Delta_E$ is a $k$-equivalence.
Then, by \Cref{thm:  equivalent defn of para tc}, we have 
$
    \dim(E) < (\sct(\Delta_E)+1) \cdot (k+1) - 1, 
$
and hence, by \cite[Theorem 14]{secat-and-relcat-II}, it follows that
$$
    \rct(\Delta_E) = \sct(\Delta_E) = \TC[p \colon E \to B].
$$
Hence, the result follows from \Cref{thm: equivalent defn of monoidal para tc}.
\end{proof}

\begin{example}
Let $\zeta \colon E \to B$ be a vector bundle of rank $2r$ admitting a complex structure.
Let $\dot{\zeta} \colon \dot{E} \to B$ be its sphere bundle.
Then $\TC[\dot{\zeta} \colon \dot{E} \to B] = 1$, see \cite[Corollary 17]{ptcspherebundles}.
Since the fibre $S^{2r-1}$ of $\dot{\zeta}$ is $(2r-2)$-connected and $\dim(\dot{E}) = (2r-1) + \dim(B)$, it follows that 
$$
    \dim(\dot{E})<(1+1)(2r-2+1)-1 \iff \dim(B) < 2(r-1).
$$
Hence,
$
    \TC^{M}[\dot{\zeta} \colon \dot{E} \to B] = 1,
$
if $\dim(B) < 2(r-1)$. 
\end{example}

It was proved in \cite[Proposition 4.3]{farber-para-tc} that for a principal $G$-bundle $p \colon E \to B$, where $G$ is a connected topological group, we have
$$
    \TC[p \colon E \to B] = \TC(G) = \ct(G).
$$
We extend this result, thereby generalizing both the above result and Dranishnikov's \cite[Lemma 2.7]{dranishnikov2014topological}, which states
$$
    \TC^{M}(G) = \TC(G) = \ct(G),
$$
for a connected Lie group $G$.

\begin{proposition}\label{prop: Iwase-Sakai conjecture for principal G-bundle}
Let $p \colon E \to B$ be a principal $G$-bundle, where $G$ is a connected Lie group.
Then
    $$
        \TC^{M}[p \colon E \to B] = \TC[p \colon E \to B] = \TC^{M}(G) = \TC(G) = \ct(G).
    $$
\end{proposition}

\begin{proof}
    It is enough to show that $\TC^{M}[p \colon E \to B] \leq \ct(G)$.
    Let 
    $$
    P_0 G = \{ \gamma \in G^{I} \mid \gamma(0) = e_G\},
    $$
    where $e_G \in G$ is the identity element.
    By \cite[Proposition 4.3]{farber-para-tc}, the following diagram is commutative, 
\[
\begin{tikzcd}
E \times P_0G \arrow[d, "\mathrm{id}_E \times p_0"'] \arrow[r, "F'"] & E^{I}_B \arrow[d, "\Pi"] \\
E \times G \arrow[r, "F"]                                            & E \times_B E            
\end{tikzcd}
\]
    where $F$ and $F'$ are homeomorphisms given by $F(e,g) = (e,ge)$ and $F'(e,\gamma)(t) = \gamma(t)e$, and $p_0$ is the fibration which sends $\gamma$ to $\gamma(1)$.
    
    Suppose $\ct(G)=n$.
    Note that the inclusion $\{e_G\} \hookrightarrow G$ is a closed cofibration.
    Hence, by \cite[Lemma 1.25]{CLOT}, there exists an open cover $\{U_0, \dots, U_n\}$ of $G$ such that $e_G \in U_i$ for all $i$ and each $U_i$ admits a homotopy $H_i \colon U_i \times I \to G$ such that 
    $$
    H_i(g,0)=e_G, \quad H_i(e_G, t) = e_G \quad \text{ and } \quad H_i(g,1) = g
    $$
    for all $g \in U_i$ and $t \in I$.
    Let $V_i := F(E \times U_i)$.
    Since $e_G \in U_i$ for all $i$, we have $\Delta E \subset V_i$ for all $i$.
    If $s_i \colon U_i \to P_0 G$ is the section of $p_0$ defined by 
    $
        s_i(g)(t) := H_i(g,t),
    $
    then 
    $$
    \widetilde{s}_i := F' \circ (\mathrm{id}_E \times s_i) \circ F^{-1} \colon V_i \to E^{I}_B
    $$
    is a section of $\Pi$.
    Further, $\widetilde{s}_i$ satisfies 
    $$
    \widetilde{s}_i(e,e) 
        = ( F' \circ (\mathrm{id}_E \times s_i) \circ F^{-1})(e,e)
        = ( F' \circ (\mathrm{id}_E \times s_i))(e,e_G)
        = F'(e,\mathrm{c}_{e_G})
        = \mathrm{c}_e.
    $$
    Hence, $\TC^{M}[p \colon E \to B] \leq n$.
\end{proof}

\begin{example}
Let $\eta \colon E(\eta) \to \C P^n$ be the canonical complex line bundle over the complex projective space $\C P^n$, viewed as a rank 2 real vector bundle.
The unit sphere bundle $\dot{\eta}$ associated with the vector bundle $\zeta$ is a principal $S^1$-bundle.
Hence, $\TC^M[\dot{\eta} \colon \dot{E}(\eta) \to \C P^n] = \ct(S^1) = 1$.

In \cite[Example 10]{ptcspherebundles}, the authors gave an explicit motion planners over a generalized cover $\{F_0,F_1\}$ of $ \dot{E}(\eta) \times_{\C P^n}  \dot{E}(\eta)$ such that $\Delta  \dot{E}(\eta) \subset F_0$.
In the following results, we show that these kinds of covers can be used to estimate monoidal parametrized topological complexity.
\end{example}

In the following theorem, we employ the technique developed in \cite[Theorem 1.4]{aguilar-gonzalez-2023motion} to generalize \cite[Corollary 3 (1)]{I-S-erratum} to the parametrized setting.

\begin{theorem}
\label{thm: mon-para-tc = para-tc under special cover}
Let $p \colon E \to B$ be an LEC fibration such that $E \times_B E$ is paracompact Hausdorff.
If $\{U_i\}_{i=0}^{k}$ is an open cover of $E \times_B E$ with sections $s_i \colon U_i \to E^{I}_B$ of $\Pi$ such that $\Delta E \subset U_0$ and $\Delta E \cap U_i = \emptyset$ for all $i \ge 1$, then
$$
    \TC^{M}[p \colon E \to B] \leq k.
$$
Moreover, if $\TC[p \colon E \to B] = k$, then 
$$
    \TC^M[p \colon E \to B] = \TC[p \colon E \to B] = k.
$$
\end{theorem}

\begin{proof}
If $k=0$, then \Cref{prop: mon-para-tc = 0 iff para-tc = 0} implies that $\TC^{M}[p \colon E \to B] = 0$.
We can therefore assume $k \geq 1$. 
By \Cref{cor: existence of monoidal section}, there is an open subset $W$ of $E \times_B E$ containing $\Delta E$ and a local section $\sigma \colon W \to E^I_B$ of $\Pi$ satisfying $\sigma(e,e) = \mathrm{c}_e$ for all $e \in E$. 
Since $E \times_B E$ is normal, there exists an open cover $\{W_i\}_{i=0}^k$ of $E \times_B E$ such that $W_i \subset \overline{W_i} \subset U_i$ for all $i \ge 0$. 
Consider the open neighbourhood $N$ of $\Delta E$ given by
$$
    N = W \cap U_0 
            \cap ((E \times_B E)\setminus \overline{W_1}) 
            \cap \cdots 
            \cap ((E \times_B E) \setminus \overline{W_n}).
$$
By normality of $E \times_B E$, we may again choose an open set $M \subset E \times_B E$ satisfying
$$
    \Delta E \subset M \subset \overline{M} \subset N .
$$
Let $V_0 = (U_0 \setminus \overline{M}) \sqcup M$.
Then $\Delta E \subset M \subset V_0$ and the section $\sigma_0 \colon V_0 \to E^{I}_{B}$
of $\Pi$ defined by
$$
\sigma_0(e_1,e_2) =
    \begin{cases}
    s_0(e_1,e_2), & \text{if } (e_1,e_2) \in U_0 \setminus \overline{M},\\
    \sigma(e_1,e_2), & \text{if } (e_1,e_2) \in M,
    \end{cases}
$$
satisfies $\sigma_0(e,e) = \sigma(e,e) = \mathrm{c}_e$ for all $e \in E$.

For $1 \leq i \leq k$, set $V_i = W_i \sqcup N$ and define the local sections $\sigma_i \colon V_i \to E^{I}_B$
of $\Pi$ by
$$
\sigma_i(e_1,e_2) =
    \begin{cases}
    s_i(e_1,e_2), & \text{if } (e_1,e_2) \in W_i,\\
    \sigma(e_1,e_2), & \text{if } (e_1,e_2) \in N.
    \end{cases}
$$
Then $\Delta E \subset N \subset V_i$ and $\sigma_i$ satisfies $\sigma_i(e,e) = \sigma(e,e) = \mathrm{c}_e$ for all $e \in E$.

We now show that the collection $\{V_0,\dots,V_k\}$ covers $E \times_B E$.
Indeed,
\begin{align*}
    V_0 \cup V_1 \cup \dots \cup V_k 
        & = ((U_0 \setminus \overline{M}) \cup M) \cup (W_1 \cup N) \cup \dots \cup (W_k \cup N) \\
        & = (U_0 \setminus \overline{M}) \cup M \cup (W_1 \cup \dots \cup W_k) \cup N \\
        & = U_0 \cup W_1 \cup \dots \cup W_k \cup N \\
        & \supseteq  W_0 \cup W_1 \cup \dots W_k \\
        & = E \times_B E,
\end{align*}
where the equality
$$
(U_0 \setminus \overline{M}) \cup M \cup (W_1 \cup \dots \cup W_k) \cup N
    \;=\; U_0 \cup W_1 \cup \dots \cup W_k \cup N
$$
follows from the fact that $M \subset \overline{M} \subset N$.
Hence, $\TC^{M}[p \colon E \to B] \leq k$.
\end{proof}

\begin{remark}
If $E$ is Hausdorff, then the condition $U_i \cap \Delta E = \emptyset$ for all $i \geq 1$ can be omitted in \Cref{thm: mon-para-tc = para-tc under special cover}, since the diagonal $\Delta E$ can be removed from the open sets $U_i$ for all $i \geq 1$.
\end{remark}

\begin{corollary}
\label{cor: mon-para-tc = para-tc under gen-special cover}
Let $p \colon E \to B$ be a locally trivial fibration, where $E$ and $B$ are metrizable separable ANRs.
Then $\TC[p \colon E \to B] = \TC_g[p \colon E \to B]$.
If $\{U_i\}_{i=0}^{k}$ is a (not necessarily open) cover of $E \times_B E$ with sections $s_i \colon U_i \to E^{I}_B$ of $\Pi$ such that $\Delta E \subset U_0$, then
$$
    \TC^{M}[p \colon E \to B] \leq k.
$$
Moreover, if $\TC[p \colon E \to B] = k$, then 
$$
    \TC^M[p \colon E \to B] = \TC[p \colon E \to B] = \TC_g[p \colon E \to B] = k.
$$
\end{corollary}

\begin{proof}
By \cite[Theorem 4.7]{farber-para-tc}, the map $\Pi$ is a fibration between ANR spaces, and $\TC[p \colon E \to B] = \TC_g[p \colon E \to B]$.
Following the proof of \cite[Theorem 2.7]{garcia2019note}, it follows that there exists an open cover $\{V_i\}_{i=0}^{k}$ of $E \times_B E$ and sections $\sigma_i \colon V_i \to E^{I}_B$ of $\Pi$ such that $U_i \subseteq V_i$ and $\left.\sigma_i\right|_{U_i} \simeq s_i$.     
Let $W_0 = V_0$ and $W_i = V_i \setminus \Delta E$. 
Then the cover $\{W_i\}_{i=0}^{k}$ admitting the sections $\left.\sigma_i\right|_{W_i}$ of $\Pi$ satisfies the hypotheses of \Cref{thm: mon-para-tc = para-tc under special cover}.
Hence, the desired result follows.
\end{proof}

\begin{example}
Let $\eta \colon E(\eta) \to \C P^n$ be the canonical complex line bundle over the complex projective space $\C P^n$.
Let $\zeta = \eta \oplus \epsilon$, where $\epsilon$ is the trivial real line  bundle over $\C P^n$.
If $\dot{\zeta}$ is the sphere bundle associated with the bundle $\zeta$, then \cite[Example 20]{ptcspherebundles} establishes that
\begin{equation}
\label{eq: para-tc of sphere bundle of can + trivial over CP^n}
    \TC[\dot{\zeta} \colon \dot{E}(\zeta) \to \C P^n] \leq n+2,
\end{equation}
for all $n$, and the equality holds for all even $n$.
Moreover, the authors also described explicit motion planners for which equation \eqref{eq: para-tc of sphere bundle of can + trivial over CP^n} holds.
We note the cover described satisfies the assumption of \Cref{cor: mon-para-tc = para-tc under gen-special cover}. 
Hence, $\TC^M[\dot{\zeta} \colon \dot{E}(\zeta) \to \C P^n] \leq n+2$ for all $n$, and the equality holds for all even $n$.
\end{example}

In \cite[Theorem 3.4]{Sequential-PTC-and-related-invariants}, Farber and Oprea showed that if $p \colon E \to B$ is a locally trivial fibre bundle with fibre $F$ and structure group $G$ such that $p$ is the associated fibre bundle to a principal $G$-bundle $\tau \colon P \to B$, then $     \TC[p \colon E \to B] \leq \TC_G(F)$, where $\TC_G(F)$ is the equivariant topological complexity of $F$, introduced by Colman and Grant in \cite{EqTC}.
Note that we always have $\TC(F) \leq \TC[p \colon E \to B]$, see \cite{farber-para-tc}.
Hence, we have the following series of inequalities
\begin{equation}
\label{eq: para-tc of associated bundle}
    \TC(F) \leq \TC[p \colon E \to B] \leq \TC_G(F),
\end{equation}
We refer the reader to \cite[Example 4.31]{Inv-Para-TC} for a few examples satisfying $\TC(F) = \TC_{G}(F)$.
We now use the proof of Farber and Oprea's result in our setup to obtain the following corollary.

\begin{corollary}
\label{cor: mon-para-tc = para-tc of associated bundle}
Let $p \colon E \to B$ be a locally trivial fibration with path-connected fibre $F$ and structure group $G$, where $E$ and $B$ are metrizable separable ANRs. 
If
\begin{enumerate}
    \item $\tau \colon P \to B$ is a principal $G$-bundle such that the associated bundle $F \times_G P \to B$ is $p$,

    \item $F \times F$ admits a $G$-invariant (not necessarily open) cover $\{U_i\}_{i=0}^{k}$ with $G$-equivariant sections $s_i \colon U_i \to F^I$ of the free path space fibration $\pi \colon F^{I} \to F \times F$ such that $\Delta F \subset U_0$,
\end{enumerate}
then
$$
    \TC^{M}[p \colon E \to B] \leq k.
$$
Moreover, if $\TC[p \colon E \to B] = k$, then
$$
    \TC^{M}[p \colon E \to B] = \TC[p \colon E \to B] = \TC_g[p \colon E \to B] = k.
$$
\end{corollary}

\begin{proof}
From \cite[Theorem 3.4]{Sequential-PTC-and-related-invariants}, we have the following commutative diagram
\[
\begin{tikzcd}
F^I \times_G P \arrow[d, "\pi \times_G \mathrm{id}_P"'] \arrow[r, "\alpha"] & E^{I}_B \arrow[d, "\Pi"] \\
(F \times F) \times_G P \arrow[r, "\beta"]                                  & {E \times_B E\,,}       
\end{tikzcd}
\]
where $\alpha$ and $\beta$ are homeomorphisms.
The homeomorphism 
$$
    \beta \colon (F \times F) \times_G P 
        \to E \times_B E 
        = (F \times_G P) \times_B (F \times_G P)
$$ 
is given by $\beta \left([f_1,f_2,p]\right) = ([f_1,p],[f_2,p])$ for $f_1,f_2 \in F$ and $p \in P$.
If we have a cover $\{U_i\}_{i=0}^k$ of $F \times F$ satisfying (3), then $\{W_i := \beta \left(U_i \times_G P \right)\}_{i=0}^{k}$ is a cover of $E \times_B E$ with sections $\alpha \circ (s_i \times_G \id_{P}) \circ \beta^{-1} \colon W_i \to E^{I}_B$ of $\Pi$ such that $\Delta E \subset W_0$.
Hence, by \Cref{cor: mon-para-tc = para-tc under gen-special cover}, the desired result follows.
\end{proof}

\begin{example}
Let $p \colon E \to B$ be a locally trivial fibration with fibre an $n$-dimensional sphere $S^n$ and structure group $\Z_2 = \langle \sigma \rangle$, where $\sigma \colon S^n \to S^n$ is the antipodal involution, and $E$ and $B$ are metrizable separable ANRs. 
If $p$ is the associated fibre bundle to a principal $\Z_2$-bundle $\tau \colon P \to B$, then, by applying equation \eqref{eq: para-tc of associated bundle} applied to the bundle $S^n \hookrightarrow E \to B$, it follows that
$$
    \TC(S^n) 
        = \TC[p \colon E \to B] 
        = \TC_{\Z_2}(S^n)
        = \begin{cases}
            1, & n \text{ odd},\\
            2, & n \text{ even},
        \end{cases}
$$
by \cite[Theorem 8]{FarberTC} and \cite[Lemma 4.1]{eqtcprodineq}.
Moreover, in \cite[Lemma 4.1]{eqtcprodineq}, the authors also gave an explicit cover for $S^n \times S^n$ of the form in \Cref{cor: mon-para-tc = para-tc of associated bundle} (3) of cardinality $\TC_{\Z_2}(S^n) +1$.
Hence, $\TC^M[p \colon E \to B] = \TC[p \colon E \to B]$.
\end{example}

\begin{example}
Suppose $M$ and $N$ are topological spaces with involutions $\tau \colon M \to M$ and $\sigma \colon N \to N$, where $\sigma$ is fixed point free.
Then the following quotient space 
$$
    X(M,N) := \frac{M \times N}{(m,n) \sim (\tau(m),\sigma(n))}
$$
is called the \emph{generalized projective product space}.
This class of spaces was introduced by Sarkar and Zvengrowski in \cite{sarkargpps}, and it includes projective product spaces \cite{Davis} and Dold manifolds \cite{Dold}. 

Suppose that $N$ is completely regular Hausdorff and $M$ is path-connected.
It was observed in \cite[Proposition 6.1]{A-D-S} that the natural map
$$
    p \colon X(M,N) \to N/\langle \sigma \rangle, \quad p([m,n]) = [n]
$$
is a fibre bundle with fibre $M$ and structure group $\langle \tau \rangle$.
Moreover, the orbit map $q \colon N \to N/\langle \sigma \rangle$ is a principal $\left < \sigma \right >$-bundle, by \cite[Chapter II, Theorem 5.8]{Bredon}.
Clearly, $p$ is the associated fibre bundle $M \times_{\Z_2} N \to N/\langle \sigma \rangle$, where $\Z_2 = \{\id_{M \times N}, \sigma \times \tau \}$. 
Hence, by equation \eqref{eq: para-tc of associated bundle}, it follows that
$$
    \TC(M) 
        \leq \TC[p \colon X(M,N) \to N/\langle \sigma \rangle] 
        \leq \TC_{\langle \tau\rangle}(M)\,.
$$

Under appropriate conditions on the spaces $M$ and $N$, if one of the following conditions is satisfied.
\begin{enumerate}
    \item $\TC(M) = \TC_{\langle \tau\rangle}(M)$ and $M \times M$ admits a cover of the form in \Cref{cor: mon-para-tc = para-tc of associated bundle} (3) of cardinality $\TC_{\langle \tau\rangle}(M) + 1$ (as in the example above).

    \item $\dim(N) \leq (\TC(M) + 1)(k+1) - 1 - \dim(M)$, where $M$ is $k$-connected.
\end{enumerate}
Then we can conclude that 
$$
    \TC^{M}[p \colon X(M,N) \to N/\langle \sigma \rangle] 
        = \TC[p \colon X(M,N) \to N/\langle \sigma \rangle]\,,
$$
by \Cref{cor: mon-para-tc = para-tc of associated bundle} and \Cref{thm: mon-para-tc = para-tc if dim(E) < (TC[p] + 1)(k+1)-1}, respectively.
We note that, since the statements above rely only on the properties of the spaces $M$ and $N$, they are easier to verify. 
\end{example}

\subsection{D-monoidal parametrized topological complexity}
\label{subsec: D-mono-para-tc}

\begin{definition}
\label{defn: D-mon-para-tc}
The \emph{D-monoidal parametrized topological complexity} of a fibration $p \colon E \to B$, denoted by $\TC^{DM}[p \colon E\to B]$, is the least nonnegative integer $k$ such that $E \times_B E$ may be covered by open subsets $U_0,\dots,U_k$, each of which admits a local section $\sigma_i \colon U_i \to E^I_B$ of $\Pi$ such that $\sigma_i(e,e) = \mathrm{c}_e$ for all $e \in E$ with $(e,e) \in U_i$.
If such an integer does not exist, we set $\TC^{DM}[p \colon E \to B] = \infty$.
\end{definition}

It is clear that one has the following chain of inequalities:
\begin{equation}
\label{eq: para-tc leq D-mon-para-tc leq mon-para-tc}
     \TC[p \colon E \to B] 
        \leq  \TC^{DM}[p \colon E \to B] 
        \leq \TC^{M}[p \colon E \to B].
\end{equation}
Similar to \Cref{defn: gen-para-tc}, we can define generalized notions of monoidal and D-monoidal parametrized topological complexities of $p$, denoted by 
$$
    \TC_g^{M}[p \colon E \to B]
       \quad \text{ and } \quad
    \TC_g^{DM}[p \colon E \to B],
$$
respectively, except that the covering sets are not required to be open.
It is clear that the following chain of inequalities holds:
\begin{enumerate}
\item 
$
    \TC_g[p \colon E \to B] 
        \leq \TC_g^{DM}[p \colon E \to B]
        \leq \TC_g^{M}[p \colon E \to B] 
        \leq \TC^{M}[p \colon E \to B],
$ and

\item 
$
    \TC_g^{DM}[p \colon E \to B] 
        \leq \TC^{DM}[p \colon E \to B].
$
\end{enumerate}
Using the following lemma, we show that, when $E \times_B E$ is a normal space, the chain of inequalities (1) and (2) can be combined into a single chain of inequalities, see \Cref{cor: gen-mon-para-tc leq d-mon-para-tc}. 
Moreover, we note that the generalized monoidal and (generalized) D-monoidal parametrized topological complexities behave well under pullbacks, in the sense of \Cref{prop: mono-para-tc under pullback}.

\begin{lemma}
\label{lemma: gen-mon-para-tc leq d-mon-para-tc}
If $p \colon E \to B$ is a fibration such that $E \times_B E$ is a normal space, then
    $$
        \TC^{M}_g[p \colon E \to B] \leq \TC^{DM}[p \colon E \to B].
    $$
\end{lemma}

\begin{proof}
    Suppose $\TC^{DM}[p \colon E \to B] = n$.
    Let $\{U_0,\dots,U_n\}$ be an open cover of $E \times_B E$ such that each $U_i$ admits a section $\sigma_i \colon U_i \to E^{I}_B$ of $\Pi$ such that $\sigma_i(e,e) = \mathrm{c}_e$ for all $e \in E$ with $(e,e) \in U_i$. 
    As $E \times_B E$ is normal, there exists a closed cover $\{Z_0,\dots,Z_n\}$ of $E \times_B E$ such that $Z_i \subseteq U_i$ for all $i$, see \cite[Section 31, Exercise 2]{munkres2000topology}.
    Define $s_i \colon Z_i \cup \Delta E \to E^I_B$ by
    $$
        s_i(x_1,x_2) = 
            \begin{cases}
                \sigma_i(e_1,e_2) & \text{ if } (e_1,e_2) \in Z_i,\\
                \mathrm{c}_e      & \text{ if } (e_1,e_2) \in \Delta E.
            \end{cases}
    $$
    Note that $s_i$ is well-defined because $\sigma_i(e,e)=\mathrm{c}_e$ whenever $(e,e) \in Z_i$. 
    Hence, $s_i$ is a continuous extension of $\left.\sigma_i\right|_{Z_i}$ such that $s_i$ is a local section of $\Pi$ satisfying $\sigma_i(e,e) = \mathrm{c}_e$ for all $e \in E$.
\end{proof}

\begin{corollary}
\label{cor: gen-mon-para-tc leq d-mon-para-tc}
If $p \colon E \to B$ is a fibration such that $E \times_B E$ is a normal space, then
\begin{align*}
    \TC_g^{DM}[p \colon E \to B] 
            \leq \TC_g^{M}[p \colon E \to B] 
            \leq \TC^{DM}[p \colon E \to B] 
            \leq \TC^{M}[p \colon E \to B].
\end{align*}
\end{corollary}

\begin{lemma}
\label{lemma: recal(Delta) leq gen-mon-para-tc}
Suppose $p \colon E \to B$ is an LEC fibration.
Then
$$
\mathrm{relcat}_g(\Delta_E)
        \leq  \TC^{M}_g[p \colon E\to B],
$$
where $\Delta_E \colon E \to E \times_B E$ is the diagonal inclusion.
\end{lemma}

\begin{proof}
Suppose $U$ is a subset of $E \times_B E$ containing the diagonal $\Delta E$, together with a local section $\sigma \colon U \to E^I_B$ of $\Pi$ satisfying $\sigma(e,e)=\mathrm{c}_e$ for all $e \in E$. 
Let $H \colon U \times I \to E \times E$ be the homotopy defined as $H(e_1,e_2,t) := (\sigma(e_1,e_2)(t),e_2)$ for all $(e_1,e_2) \in U,\, t \in I$.
As $\sigma(e_1,e_2) \in E^{I}_B$, it follows that $p(\sigma(e_1,e_2)(t)) = b$ for some $b \in B$ and for all $t \in I$.
In particular, $p(\sigma(e_1,e_2)(t)) =p(\sigma(e_1,e_2)(1)) = p(e_2)$ implies that the image of $H$ lies in $E \times_B E$.
Moreover, $H$ satisfies $H(e_1,e_2,0) = (e_1,e_2)$, $H(e_1,e_2,1) = (e_2,e_2) \in \Delta_E(E)$ and $H(e,e,t) = (e,e)$.
Extending this argument to covers, the desired inequality follows.
\end{proof}

We now adapt the ideas presented in \cite[Remark 2.19]{garcia2019note} to present the main result of this section.
This theorem generalizes \cite[Proposition 2.3]{aguilar-gonzalez-2023motion}, and \cite[Corollary 2.18]{garcia2019note} to the parametrized setting.

\begin{theorem}
\label{thm: mon-para-tc = D-mon-para-tc}
Suppose $p \colon E \to B$ is a fibration, where $E$ and $E \times_B E$ are ANRs.
Then
$$
    \TC^{M}[p \colon E\to B]
        = \TC_g^{M}[p \colon E\to B]
        = \TC^{DM}[p \colon E\to B].
$$
\end{theorem}

\begin{proof}
The series of inequalities
\begin{align*}
    \TC^{M}[p \colon E\to B] 
            & = \mathrm{relcat}(\Delta_E) 
                & \mathrm{\Cref{thm: equivalent defn of monoidal para tc}} \\
            & = \mathrm{relcat}_g(\Delta_E)
                & \mathrm{\Cref{thm: relcat = gen-relcat}} \\
            & \leq \TC_g^{M}[p \colon E\to B] 
                & \mathrm{\Cref{lemma: recal(Delta) leq gen-mon-para-tc}} \\
            & \leq \TC^{DM}[p \colon E\to B] 
                & \mathrm{\Cref{lemma: gen-mon-para-tc leq d-mon-para-tc}}\\
            & \leq \TC^{M}[p \colon E\to B]
\end{align*}
implies the desired result.
\end{proof}

We note that results of the form \Cref{thm: mon-para-tc = para-tc under special cover} and \Cref{cor: mon-para-tc = para-tc under gen-special cover} also hold for D-monoidal parametrized topological complexity, since the inequality \eqref{eq: para-tc leq D-mon-para-tc leq mon-para-tc}, and its generalized version, holds.

\subsection{Fadell–Husseini monoidal topological complexity}
\label{subsec: FM-mono-para-tc}

\begin{definition}
\label{defn: FH-mon-para-tc}
The \emph{Fadell-Husseini monoidal parametrized topological complexity} of a fibration $p \colon E \to B$, denoted by $\TC^{FH}[p \colon E \to B]$, is the least nonnegative integer $k$ such that $E \times_B E$ may be covered by open subsets $U_0,\dots,U_k$, each of which admits a local section $\sigma_i \colon U_i \to E^I_B$ of $\Pi$ such that 
\begin{enumerate}
    \item the diagonal $\Delta E$ is a subset of $U_0$,
    \item $\sigma_0(e,e) = \mathrm{c}_e$ for all $e \in E$, 
    \item $\Delta E \cap U_i = \emptyset$ for all $i \geq 1$.
\end{enumerate}    
If such an integer does not exist, we set $\TC^{FH}[p \colon E \to B] = \infty$.
The generalized notion of the Fadell–Husseini monoidal parametrized topological complexity of $p$, denoted $\TC_g^{FH}[p \colon E \to B]$, is defined in the same way, except that the covering sets are not required to be open.
\end{definition}

It is clear that one has the following inequalities:
\begin{enumerate}
    \item $\TC^{DM}[p \colon E \to B] \leq \TC^{FH}[p \colon E \to B]$, 
    \item $\TC_{g}^{FH}[p \colon E \to B] \leq \TC^{FH}[p \colon E \to B]$,
    \item $\TC_{g}^{DM}[p \colon E \to B] \leq \TC_{g}^{FH}[p \colon E \to B] \leq \TC_{g}^{M}[p \colon E \to B]$.
\end{enumerate}
Here, the second inequality in (3) follows from the fact that the diagonal $\Delta E$ can be removed from the open sets $U_i$ for $i \geq 1$.
Hence, if $E$ is Hausdorff, we also have
$$
    \TC^{FH}[p \colon E \to B] \leq \TC^{M}[p \colon E \to B].
$$

The proof of the following lemma is similar to that of \Cref{lemma: recal(Delta) leq gen-mon-para-tc} and is therefore omitted.

\begin{lemma}
\label{lemma: FH-recal(Delta) leq gen-FH-mon-para-tc}
Suppose $p \colon E \to B$ is an LEC fibration.
Then
$$
\mathrm{relcat}^{FH}_g(\Delta_E)
        \leq  \TC^{FH}_g[p \colon E\to B].
$$
where $\Delta_E \colon E \to E \times_B E$ is the diagonal inclusion.
\end{lemma}

\begin{theorem}
\label{thm: mon-para-tc = FH-mon-para-tc}
Suppose $p \colon E \to B$ is a fibration, where $E$ and $E \times_B E$ are ANRs.
Then
$$
    \TC_g^{FH}[p \colon E\to B]
        = \TC^{FH}[p \colon E \to B]
        = \TC^{M}[p \colon E\to B].
$$
\end{theorem}

\begin{proof}
The series of inequalities
\begin{align*}
    \mathrm{relcat}^{FH}_g(\Delta_E)
            & \leq  \TC^{FH}_g[p \colon E\to B]
                & \text{\Cref{lemma: FH-recal(Delta) leq gen-FH-mon-para-tc}} \\
            & \leq \TC^{FH}[p \colon E\to B] \\
            & \leq \TC^{M}[p \colon E\to B] \\
            & = \rct_g(\Delta_E)
                & \text{\Cref{thm: equivalent defn of monoidal para tc} and \Cref{thm: relcat = gen-relcat}}\\
            & = \rct_g^{FH}(\Delta_E)
                & \text{\Cref{thm: gen-relcat = gen-FH-relcat}}
\end{align*}  
implies the desired result.
\end{proof}

When the space $E$ is Hausdorff, results of the form \Cref{thm: mon-para-tc = para-tc under special cover} and \Cref{cor: mon-para-tc = para-tc under gen-special cover} also hold for Fadell-Husseini monoidal parametrized topological complexity, since the following inequality holds:
$$
    \TC[p \colon E \to B] \leq \TC^{FH}[p \colon E \to B] \leq \TC^{M}[p \colon E \to B].
$$
As a result, Theorem 1.4 in \cite{aguilar-gonzalez-2023motion} extends to the parametrized setting.

\section{Symmetrized parametrized topological complexity}
\label{sec: symm-para-tc}

In this section, we introduce the notion of the symmetrized parametrized topological complexity of a fibration.
Suppose $p \colon E \to B$ is a fibration.
Then there is an action of $\Z_2 = \langle g \rangle$ on the spaces $E^I_B$ and $E\times_B E$ defined as follows: 
\[
    g\cdot \gamma(t)=\gamma(1-t), \quad \text{and} \quad g\cdot (e_1,e_2)=(e_2,e_1)\,,
\]
for $\gamma\in E^I_B$ and $(e_1,e_2)\in E\times_B E$.
The fibration $\Pi \colon E^I_B\to E\times_B E$ satisfies
\[
    \Pi(g \cdot \gamma) 
        = (g \cdot \gamma(0),g \cdot \gamma(1)) 
        = (\gamma(1),\gamma(0))
        = g \cdot (\gamma(0),\gamma(1))\,,
\]
and hence the map $\Pi$ is $\Z_2$-equivariant.
In fact, we now show that $\Pi$ is a $\Z_2$-fibration.

\begin{proposition}
Suppose $p \colon E \to B$ is a fibration.
Then the evaluation map
\begin{equation*}
    \Pi \colon E^{I}_B \to E \times_B E, \quad 
        \alpha \mapsto (\alpha(0),\alpha(1))
\end{equation*}
is a $\Z_2$-fibration.
\end{proposition}

\begin{proof}
Let $\Z_2 = \langle g \rangle$ act on the unit interval $I$ via the involution $g \cdot t = 1-t$ for $t \in I$. 
It was shown in \cite[Example 2.4]{Grantsymmtc} that the inclusion map $i \colon \{0,1\} \hookrightarrow I$ is a $\Z_2$-cofibration.
Hence, by \Cref{thm: restriction map is a G-fibration}, it follows that $\Pi$ is a $\Z_2$-fibration.
\end{proof}

\begin{notation*}
Throughout this section, we denote the group $\Z_2$ by $G$.
\end{notation*}

\begin{remark}
\label{rem: fibre of the fixed point fibration for parameterized fibration}
    Note that the fibre of $\Pi$ over the point $(e,e) \in E \times_B E$ is the based loop space $\Omega F$ of the pointed space $(F,e)$, where $F = p^{-1}(b)$ and $b = p(e)$.
    Since $(e,e) \in (E \times_B E)^G = \Delta E$, it follows from \Cref{prop: fibre of fixed point fibration} that $\Omega F$ admits a $G$-action, given by $g \cdot \gamma = \overline{\gamma}$, and the induced fibration on the fixed points 
    $$
        \Pi^G \colon (E^{I}_B)^G \to \Delta E
    $$
    has the fibre $(\Omega F)^G$, which is homeomorphic to the based path space $P_0 F = \{\gamma \in F^I \mid \gamma(0) = e\}$.
\end{remark}

We are now in a position to define our main object.

\begin{definition}
    The \emph{symmetrized parametrized topological complexity} of a fibration $p \colon E\to B$, denoted by $\TC^{\Sigma}[p \colon E\to B]$, is defined as 
    \[
        \TC^{\Sigma}[p \colon E\to B] := \sct_G(\Pi).
    \]
\end{definition}

\subsection{Properties}

We first establish the fibrewise homotopy invariance of $\mathrm{TC}^{\Sigma}[p\colon E\to B]$. 
The notion of fibre-preserving homotopy equivalence is recalled in \Cref{sec: mono-para-tc}.

\begin{proposition}
\label{prop: fibrewise homotopy invariance of sym-para-tc}
Suppose $p \colon E \to B$ and $p' \colon E' \to B'$ are fibrations.
If there exists homotopy equivalences $f \colon E \to E'$ and $g \colon B \to B'$ such that $(f,g) \colon p \to p'$ is a fibre-preserving map, then 
$$
    \TC^{\Sigma}[p \colon E \to B] = \TC^{\Sigma}[p' \colon E' \to B'].
$$
\end{proposition}

\begin{proof}
By \cite[Page 53]{JPMay}, it follows that $(f,g)$ is in fact a fibre-preserving homotopy equivalence.
Moreover, following \cite[Proposition 2.6]{GrantPTC}, there exists fibre-preserving map $(h,h') \colon \Pi \to \Pi'$ between their parametrized endpoint fibrations
\[
\begin{tikzcd}
E_B^I \arrow[r, "h"] \arrow[d, "\Pi"'] & (E')_{B'}^I \arrow[d, "\Pi'"] \\
E\times_B E \arrow[r, "\bar h"']       & E'\times_{B'}E'              
\end{tikzcd}
\]
such that both $h$ and $\bar h$ are $G$-homotopy equivalences. 
Hence, by \cite[Proposition 2.13]{Inv-Para-TC}, the desired equality follows.
\end{proof}

The proof of the following proposition is similar to that of \Cref{prop: mono-para-tc under pullback} and is left to the reader. 
We note that the proof of the ``in particular'' statement uses the fact that symmetrized topological complexity is a homotopy invariant.

\begin{proposition}
\label{prop: sym-para-tc under pullback}
Let $p \colon E \to B$ be a fibration with fibre $F$ and $B' \subseteq B$. 
If $E' = p^{-1}(B')$ and $p' \colon E' \to B'$ is the restricted fibration over $B'$, then 
\begin{equation*}
\label{eq: restricted fibration ineq for sym-para-tc}
\TC^{\Sigma}[p' \colon E' \to B'] \leq \TC^{\Sigma}[p \colon E \to B]\,. 
\end{equation*}
In particular, $\TC^{\Sigma}(F) \leq \TC^{\Sigma}[p \colon E \to B]$.
\end{proposition}

\begin{proposition}
\label{prop: sym-para-tc of projection map}
Let $\pr_1 \colon B \times F \to B$ be the projection map.
Then 
$$
    \TC^{\Sigma}[\pr_1 \colon B \times F \to B] = \TC^{\Sigma}(F).
$$   
\end{proposition}

\begin{proof}
Denote $E = B\times F$ and $p = \pr_1$.
Note that $E^I_B = \{(\gamma_1,\gamma_2)\in B^I\times F^I \mid \gamma_1 = \mathrm{c}_b \text{ for some } b\in B\}$ and $E\times_B E=\{(b_1,f_1,b_2,f_2)\in E\times E \mid b_1=b_2\}$.  
Consider the following diagram:
\[\begin{tikzcd}
B\times F^I \arrow[d, "\mathrm{id}_B \times \pi_F"'] \arrow[r, "\phi"] & E^{I}_B \arrow[d, "\Pi"] \\
B \times F\times F \arrow[r, "\psi"] & E \times_B E,           
\end{tikzcd}
\]   
where $\phi(b,\gamma) = (\mathrm{c}_b,\gamma)$ and $\psi(b,f_1,f_2)=(b,f_1,b,f_2)$. 
Note that $\phi$ and $\psi$ are $G$-equivariant homeomorphisms, where $G$ acts trivially on $B$, by path reversal on $F^I$, and by transposition on $F \times F$. 
Therefore,
$$
    \TC^{\Sigma}[p\colon E\to B]
        = \sct_G(\Pi)
        = \sct_G(\id_B\times \pi_F)
        = \sct_G(\pi_F)
        =\TC^{\Sigma}(F)
$$
implies the desired equality.
\end{proof}

\begin{proposition}
Let $p \colon E \to B$ be a fibration with fibre $F$.
Suppose that either 
\begin{enumerate}
\item[(a)] $p$ is a locally trivial fibration and $B$ is paracompact, or 

\item[(b)] $E$ and $B$ have homotopy type of CW complexes.
\end{enumerate}
If $F$ is contractible, then $\TC^{\Sigma}[p \colon E \to B] = 0$.
\end{proposition}

\begin{proof}
\textbf{Step 1:} (a) Suppose $H \colon F \times I \to F$ is a homotopy such that $H(f,0)=f$ and $H(f,1)=f_0$ for some $f_0 \in F$. 
Suppose $\{U_i\}_{i \in I}$ is a trivializing open cover of $B$ with trivializations $\phi_i \colon p^{-1}(U_i) \to U_i\times F$ over $U_i$. 
We claim that $p$ is a homotopy equivalence over $B$.
By \cite[Proposition 7.31 and Theorem 7.57]{james1984topology}, it is enough to show that the restriction $\left.p\right|_{p^{-1}(U_i)} \colon p^{-1}(U_i) \rightarrow U_i$ is a homotopy equivalence over $U_i$ for each $i$.
This is equivalent to showing that $\pi_1 = \left.p\right|_{p^{-1}(U_i)} \circ \phi_i^{-1} \colon U_i \times F \to F$ is a homotopy equivalence over $U_i$ for each $i$. 
Define $g_i \colon U_i \to U_i \times F$ by $g_i(u) = (u,f_0)$ for $u \in U_i$.
The map $g_i$ satisfies $\pi_1 \circ g_i = \id_{U_i}$ and $(g_i \circ \pi_1)(u,f) = (u,f_0)$.
Then $\widetilde{H}_i \colon U_i \times F \times I \to U_i \times F$ defined by $\widetilde{H}(u,f,t) = (u,H(f,t))$ is a homotopy over $U_i$ between $\id_{U_i \times F}$ and $g_i$.
Hence, $p$ is a homotopy equivalence over $B$.

(b) Since $p$ is a fibration and $F$ is contractible, by the long exact sequence of homotopy groups, it follows that
$
    p_* \colon \pi_n(E) \to \pi_n(B)
$
is an isomorphism for every $n\ge 0$. 
Therefore, $p$ is a weak homotopy equivalence.
Since $E$ and $B$ are homotopy equivalent to CW complexes, Whitehead's theorem implies that $p$ is a homotopy equivalence.
Then, by \cite[Proposition, Page 52]{JPMay}, $p$ is a fibre homotopy equivalence.\\

\textbf{Step 2:}
Hence, there exists a fibre-preserving map $\sigma \colon B \to E$, i.e., $p \circ \sigma = \id_B$, satisfying $p \circ \sigma \simeq_B \id_B$ and $\sigma \circ p \simeq_B \id_E$.
Thus there exists a homotopy $H \colon E \times I \to E$ such that $H_0 = \id_E, H_1 = \sigma \circ p$ and $p \circ H_t = p$ for all $t \in I$.
Define $s \colon E \times_B E \to E^I_B$ by
$$
    s(e_1,e_2)(t) 
        = \begin{cases}
            H(e_1,2t), & 0 \leq t \leq 1/2,\\
            H(e_2,2-2t), & 1/2 \leq t \leq 1\,.
        \end{cases}
$$
Note that $s$ is well-defined, since $p \circ H_t = p$ for all $t \in I$ implies the image of $s$ lies in $E^{I}_B$, and
$
    H(e_1,1) 
        = \sigma(p(e_1)) 
        = \sigma(p(e_2))
        = H(e_2,1).
$
Then $s$ satisfies $s(e_2,e_1)(t) = s(e_1,e_2)(1-t)$, i.e., $s(g \cdot (e_1,e_2)) = g \cdot s(e_1,e_2)$, where $\Z_2 = G = \langle g \rangle$.
Hence, $s$ is $G$-equivariant.
\end{proof}

\subsection{Bounds}

We first show that the symmetrized parametrized topological complexity always dominates the usual parametrized topological complexity.

\begin{proposition}
\label{prop: para-tc leq sym-para-tc}
Suppose $p \colon E \to B$ is a fibration. 
Then
    $$
        \TC[p \colon E\to B] \leq \TC^{\Sigma}[p \colon E\to B].
    $$
\end{proposition}

\begin{proof}
Note that we have
    \begin{align*}
        \TC[p \colon E\to B]
            = \sct(\Pi) 
            \leq \sct_G(\Pi) 
            = \TC^{\Sigma}[p \colon E\to B],
    \end{align*}
where the inequality follows from \cite[Corollary 2.8 (2)]{Inv-Para-TC}.
\end{proof}

\begin{lemma}
\label{lemma: eq def symmetrized}
Let $p\colon E\to B$ be a fibration. 
For a $G$-invariant (not necessarily open) subset $U$ of $E \times_B E$ the following are equivalent:
\begin{enumerate}
\item there exists a $G$-equivariant section of $\Pi \colon E^I_B\to E\times_B E$ over $U$.

\item there exists a $G$-homotopy between the inclusion map $\iota_U\colon U \hookrightarrow E \times_B E$ and a $G$-map $f \colon U \to E \times_B E$ which takes values in the diagonal $\Delta E$.
\end{enumerate}
\end{lemma}

\begin{proof}
$(1) \implies (2)$
Suppose $\sigma \colon U \to E^I_B$ is a $G$-equivariant section of $\Pi$. 
Then define $H \colon U\times I\to E\times_B E$ by 
$$
    H(e_1,e_2,t) := \left(\sigma(e_1,e_2)(t/2), \sigma(e_1,e_2)(1-t/2)\right).
$$
Since $(p\circ \sigma(e_1,e_2))(t) = b$ for some $b \in B$ and for all $t \in I$, the map $H$ is well defined.
Since $\sigma$ is $G$-equivariant, it follows that
\begin{align*}
    H(e_2,e_1,t) 
        & = (\sigma(e_2,e_1)(t/2), \sigma(e_2,e_1)(1-t/2)) \\
        & = ((g \cdot \sigma(e_1,e_2))(t/2), (g \cdot \sigma(e_1,e_2))(1-t/2)) \\
        & = (\sigma(e_1,e_2)(1-t/2), \sigma(e_1,e_2)(t/2)) \\
        & = g\cdot H(e_1,e_2,t)\,,
\end{align*}
which shows that the homotopy $H$ is $G$-equivariant.
Clearly, $H(e_1,e_2,0) = (e_1,e_2)$ and $H(e_1,e_2,1) = (\sigma(e_2,e_1)(1/2), \sigma(e_2,e_1)(1/2))\in \Delta E$.  

$(2) \implies (1)$
Suppose $H \colon U \times I \to E \times_{B} E$ is a $G$-homotopy between $f$ and the inclusion map $\iota_U$. 
Define a map $\sigma \colon U \to E^{I}_B$ by 
$$
    \sigma(e_1,e_2)
        = \mathrm{c}_{\pi_1(f(e_1,e_2))}
        = \mathrm{c}_{\pi_2(f(e_1,e_2))}\,,
$$ 
where $\pi_i \colon E \times_B E \to E$ denotes the projection onto the $i$-th factor. 
Since $f$ is $G$-equivariant and takes values in the diagonal $\Delta E$, we have 
$$
    f(e_2,e_1) 
        = f( g \cdot (e_1,e_2))
        = g \cdot f(e_1,e_2)
        = f(e_1,e_2)\,.
$$
We now show that $\sigma$ is $G$-equivariant. 
It suffices to verify that
$$
\sigma(e_2,e_1)(t)
    = g\cdot \sigma(e_1,e_2)(t)
    = \sigma(e_1,e_2)(1-t)
    \quad \text{for all } t\in I .
$$
Since $\sigma(e_1,e_2)$ is a constant path, it is enough to prove that $\sigma(e_2,e_1) = \sigma(e_1,e_2)$.
Indeed,
$$
    \sigma(e_1,e_2) 
        = \mathrm{c}_{\pi_1(f(e_1,e_2))}
        = \mathrm{c}_{\pi_1(f(e_2,e_1))}
        = \sigma(e_2,e_1)
$$
which proves that $\sigma$ is $G$-equivariant.
By the $G$-homotopy lifting property of $\Pi$, there exists a $G$-homotopy
$
    \widetilde{H} \colon U \times I \to E^{I}_B
$
such that the following diagram
\[
\begin{tikzcd}
    U \times \{0\} \arrow[d, hook] \arrow[r, "\sigma"] 
    & E^{I}_B \arrow[d, "\Pi"] \\
    U \times I \arrow[r, "H"] \arrow[ru, dashed, "\widetilde{H}"]
    & E \times_{B} E
\end{tikzcd}
\]
commutes. 
Hence, $\widetilde{H}_1$ is a $G$-equivariant section of $\Pi$ over $U$, since
$\Pi \circ \widetilde{H}_1 = H_1 = \iota_U$.
\end{proof}

For a fibration $p \colon E\to B$, define the \emph{fibrewise symmetric square} of $E$ as the quotient
$$
    SP^2_B(E) := \frac{E\times_B E}{G}\,.
$$
Let $\rho \colon E\times_B E \to SP^2_B(E)$ be the orbit map. 
Denote by $dE_B := \rho(\Delta E)$ the image of the diagonal under $\rho$. With these notations, we now establish a cohomological lower bound for the symmetrized parametrized topological complexity. We note that this bound is a parametrized version of \cite[Theorem 4.6]{Grantsymmtc}, which provides a cohomological lower bound for symmetrized topological complexity.

\begin{theorem}
\label{thm: cohomological lower bound on sym-para-tc}
Suppose $p \colon E \to B$ is a fibration. 
If there exists cohomology classes $u_1,\dots,u_k \in H^*(SP^2_B(E);R)$ (for any commutative ring $R$) such that
    \begin{enumerate}
        \item $u_i$ restricts to zero in $H^*(dE_B;R)$ for $i=1,\dots,k$;
        \item $u_1 \smile \dots \smile u_k \neq 0$ in $H^*(SP^2_B(E);R)$, 
    \end{enumerate}
    then $\TC^{\Sigma}[p \colon E \to B]\geq k.$   
\end{theorem}

\begin{proof}
Suppose there are cohomology classes $u_1,\dots, u_k \in H^*(SP^2_B(E);R)$ such that $u_1 \smile \dots \smile u_k \neq 0$. 
On the contrary, assume that $\TC^{\Sigma}[p \colon E \to B]<k$. 
Then $E\times_B E$ can be covered by $G$-invariant open sets $U_1,\dots, U_k$ such that there exist sections $\sigma_i \colon U_i\to E^I_B$ 
of $\Pi$ for $1\leq i\leq k$.
It follows from \Cref{lemma: eq def symmetrized} that, for each $1\leq i\leq k$, there exists a $G$-homotopy $H_i\colon U_i\to E\times_B E$ from the inclusion map $\iota_{U_i} \colon U_i \hookrightarrow E \times_B E$ to a $G$-map $f_i \colon U_i\to E\times_B E$ which takes values in $\Delta E$. 
Let $\overline{U}_i:=\rho(U_i)$, and consider the induced homotopy $\overline{H}_i \colon \overline{U}_i\times I\to SP^2_B(E)$. 
Note that $\overline{H}_i$ is a homotopy between $\iota_{\overline{U}_i}$ and $\overline{f}_i$. Moreover, $\overline{f}_i$ takes values in $dE_B$. 

Now using the long exact sequence for cohomology of the pair $(SP^2_B(E),\overline{U}_i)$ and the fact that $u_i$ restricts to zero in $dE_B$, we have cohomology classes $\bar{u}_i\in H^*(SP^2_B(E),\overline{U}_i;R)$ whose image is $u_i$ under the co-boundary map. 
Note that 
$$
    \bar{u}_1\smile \dots \smile \bar{u}_i
        \in H^*(SP^2_B(E),\cup_{i=1}^k\overline{U}_i;R)
        = H^*(SP^2_B(E),SP^2_B(E);R) 
        = 0.
$$ 
Then, by the naturality of cup products, we have $u_1\smile \dots \smile u_k = 0$. 
This is a contradiction to our assumption that $u_1\smile \dots \smile u_k\neq 0$.    
\end{proof}

The dimension–connectivity upper bound for the symmetrized topological complexity was given by Grant in \cite[Theorem 4.2]{Grantsymmtc}. 
In the following theorem, we prove its parametrized analogue, strengthened by replacing the dimension with the homotopy dimension.

\begin{theorem}
\label{thm: dimension-connectivity upper bound on symmetrized TC}
Suppose $p \colon E \to B$ is a fibration with fibre $F$ such that $E \times_B E$ is a $G$-CW complex of dimension at least 2.
If $F$ is $s$-connected, then
$$
    \TC^{\Sigma}[p \colon E \to B] 
        < \frac{\hdim_G(E \times_B E)+1}{s+1}.
$$
\end{theorem}

\begin{proof}
Note that the fibre $\Omega F$ of $\Pi$ is $(s-1)$-connected.
By \Cref{rem: fibre of the fixed point fibration for parameterized fibration}, the fixed point subspace $(\Omega F)^G$ is homeomorphic to the based path space $P_0 F = \{\gamma \in F^I \mid \gamma(0) = e\}$, which is contractible.
Hence, it follows that 
$$
    \TC^{\Sigma}[p \colon E \to B] = \sct_G(\Pi) < \frac{\hdim_G(E \times_B E)+1}{s+1}.
$$
by \cite[Theorem 2.16]{Inv-Para-TC}.
\end{proof}

\section{Monoidal symmetrized parametrized topological complexity}
\label{sec: mon-symm-para-tc}

Monoidal symmetrized topological complexity was introduced and studied by Grant in \cite{Grantsymmtc}. In this section, we introduce its parametrized analogue, along with another version in the sense of Dranishnikov and the Fadell–Husseini approach, and study their interactions with each other.
\begin{definition}
The \emph{monoidal symmetrized parametrized topological complexity} of a fibration $p \colon E \to B$, denoted by $\TC^{M,\,\Sigma}[p \colon E\to B]$, is the least nonnegative integer $k$ such that $E \times_B E$ may be covered by $G$-invariant open subsets $U_0,\dots,U_k$, each of which contains the diagonal $\Delta E$ and admits a local $G$-section $\sigma_i \colon U_i \to E^I_B$ of $\Pi$ such that $\sigma_i(e,e) = \mathrm{c}_e$ for all $e \in E$.
If such an integer does not exist, we set $\TC^{M,\, \Sigma}[p \colon E \to B] = \infty$.
\end{definition}

\begin{definition}
The \emph{D-monoidal symmetrized parametrized topological complexity} of a fibration $p \colon E \to B$, denoted by $\TC^{DM,\,\Sigma}[p \colon E\to B]$, is the least nonnegative integer $k$ such that $E \times_B E$ may be covered by $G$-invariant open subsets $U_0,\dots,U_k$, each of which admits a local $G$-section $\sigma_i \colon U_i \to E^I_B$ of $\Pi$ such that $\sigma_i(e,e) = \mathrm{c}_e$ for all $e \in E$ with $(e,e) \in U_i$.
If such an integer does not exist, we set $\TC^{DM,\,\Sigma}[p \colon E \to B] = \infty$.
\end{definition}

\begin{definition}
The \emph{Fadell-Husseni monoidal symmetrized parametrized topological complexity} of a fibration $p \colon E \to B$, denoted by $\TC^{FH,\,\Sigma}[p \colon E\to B]$, is the least nonnegative integer $k$ such that $E \times_B E$ may be covered by $G$-invariant open subsets $U_0,\dots,U_k$, each of which admits a local $G$-section $\sigma_i \colon U_i \to E^I_B$ of $\Pi$ such that 
\begin{enumerate}
    \item the diagonal $\Delta E$ is a subset of $U_0$,
    \item $\sigma_0(e,e) = \mathrm{c}_e$ for all $e \in E$, 
    \item $\Delta E \cap U_i = \emptyset$ for all $i \geq 1$.
\end{enumerate}  
If such an integer does not exist, we set $\TC^{FH,\,\Sigma}[p \colon E \to B] = \infty$.    
\end{definition}

\subsection{Properties}

It is clear that one has the following chain of inequalities:
\begin{itemize}

\item $\TC^{\Sigma}[p \colon E\to B] \leq \TC^{DM,\,\Sigma}[p \colon E\to B] \leq \TC^{FH,\,\Sigma}[p \colon E\to B]$. 

\

\item $\TC^{DM,\,\Sigma}[p \colon E\to B] \leq \TC^{M,\,\Sigma}[p \colon E\to B]$.

\

\item $\TC^{M}[p \colon E\to B] \leq \TC^{M,\,\Sigma}[p \colon E\to B]$.
\end{itemize}
If $E$ is Hausdorff, then we also have
$$
    \TC^{FH,\,\Sigma}[p \colon E\to B] \leq \TC^{M,\,\Sigma}[p \colon E\to B].
$$

We now present a lemma that will help us show that all these versions of monoidal symmetrized parametrized topological complexity agree with symmetrized parametrized topological complexity for a locally trivial fibration whose total space is an ENR and base space is an ANR.
If $E \times_B E$ is paracompact, then by \cite[Proposition 3.4]{Grantsymmtc}, we have $\TC^{\Sigma}[p \colon E \to B] \leq k$ if and only if the $(k+1)$-fold join 
$$
    \Pi_k \colon J^k_{E \times_B E}(E^{I}_B) \to E\times_B E
$$
admits a global $G$-section. 
Note that the canonical $G$-map $s_0 \colon E \to E^I_B$, which maps $e \in E$ to the constant path $\mathrm{c}_e$ at $e$, induces a $G$-map $s_k \colon E \to J^k_{E \times_B E}(E^{I}_B)$ given by
$$
    s_k(e) 
        = \left[c_e,\frac{1}{k+1},\dots,c_e,\frac{1}{k+1}\right],
$$
where $G$ acts trivially on $E$.

\begin{lemma}
\label{lemma: sections restricted to diagonal are homotopic}
Suppose $p \colon E \to B$ is a fibration with fibre $F$, where $E$ is an ANR. 
If $\sigma \colon E \times_B E \to J^{k}_{E \times_B E}(E^I_B)$ is a $G$-section of $\Pi_k$, then $\sigma \circ \Delta_E \colon E \to J^{k}_{E \times_B E}(E^{I}_B)$ and $s_k$ are $G$-homotopic, where $G$ acts trivially on $E$.
\end{lemma}
    
\begin{proof}
As $\sigma$ is $G$-equivariant, we have
$$
    \sigma (\Delta E)
        = \sigma\left((E \times_B E)^G\right) 
        \subseteq \left( J^{k}_{E \times_B E}(E^{I}_B) \right)^G
        = J^k_{\Delta E}\left((E^{I}_B)^G \right).
$$
Thus, both $\left.\sigma\right|_{\Delta E}$ and $s_k \circ \Delta_E^{-1}$ are global sections of the fibration
$
    \Pi_k^G \colon J^k_{\Delta E}((E^{I}_B)^G) \to \Delta E
$
obtained by restricting $\Pi_k$ to the fixed-point sets, where $\Delta_E^{-1} \colon \Delta E \to E$ is the inverse of the restricted diagonal map. 
Observe that the fibre of $\Pi_k^G$ is 
$$
    (J^k(\Omega F))^G = J^k((\Omega F)^G) = J^k(P_0 F),
$$
see \Cref{rem: fibre of the fixed point fibration for parameterized fibration}.
Hence, it follows that the fibre of $\Pi_k^G$ is contractible.
Since $E$ is paracompact and locally contractible \cite[Chapter IV, Corollary 3.3]{borsuk-retracts}, it follows by \cite[Proposition 2.2]{petar-trivial-fibrations} that $\Pi_k^G$ is fibre-homotopically trivial. 
Hence, in particular, any two of its sections are homotopic, and
$$
    \left.\sigma\right|_{\Delta E} \simeq s_k \circ     \Delta_E^{-1} \colon \Delta E \to J^k_{\Delta E}((E^{I}_B)^G).
$$
Hence, $\sigma \circ \Delta_E \simeq s_k \colon E \to J^k_{\Delta E}((E^{I}_B)^G)$.
The homotopy may be regarded as a $G$-homotopy with $G$ acting trivially.
\end{proof}

The following theorem generalizes \cite[Theorem 5.2]{Grantsymmtc} to the parametrized setting.

\begin{theorem}
\label{thm: equivalent defn of sym monoidal para tc}
Suppose $p \colon E \to B$ is a locally trivial fibration, where $E$ is an ENR and $B$ is a separable ANR.
Then the following statements are equivalent:
\begin{enumerate}
    \item $\TC^{M,\,\Sigma}[p \colon E\to B] \leq k$.

    \item $\TC^{FH,\,\Sigma}[p \colon E\to B] \leq k$.
        
    \item $\TC^{DM,\,\Sigma}[p \colon E\to B] \leq k$.

    \item $\TC^{\Sigma}[p \colon E\to B] \leq k$.
        
    \item The $G$-fibration $\Pi_k \colon J^k_{E \times_B E}(E^I_B) \to E \times_B E$ admits a global $G$-section.
        
    \item The $G$-fibration $\Pi_k \colon J^k_{E \times_B E}(E^I_B) \to E \times_B E$ admits a global $G$-section $\sigma$ satisfying $\sigma \circ \Delta_E = s_k$.
\end{enumerate}
\end{theorem}

\begin{proof}
    Clearly, $(1) \implies (2)$, $(2) \implies (3)$ and $(3) \implies (4)$. 

    $(4) \iff (5)$
    Note that $E \times E$ is metrizable, and hence paracompact. 
    Hence, $E \times_B E$ is also paracompact, being a closed subspace of a paracompact space.
    Thus, by \cite[Proposition 3.4]{Grantsymmtc}, we have $\TC^{\Sigma}[p \colon E\to B] \leq k$ if and only if there exists a global $G$-section $\sigma$ of the fibration $\Pi_k$.

$(5) \implies (6)$ Suppose there exists a global $G$-section $\sigma$ of the fibration $\Pi_k$.
By \Cref{lemma: sections restricted to diagonal are homotopic}, $\sigma \circ \Delta_E$ is $G$-homotopic to $s_k$, where $G$ acts trivially on $\Delta E$.
Moreover, note that the inclusion $\Delta E \hookrightarrow E \times_B E$ is a $G$-cofibration,  see \Cref{prop: cofib E to E times B E}.
The rest of the argument is identical to that in the case $(3) \implies (4)$ of \Cref{thm: equivalent defn of monoidal para tc}, generalized to the equivariant setting.

$(6) \implies (1)$ The argument is the same as the case $(4) \implies (1)$ of \Cref{thm: equivalent defn of monoidal para tc}; we only have to note that the join coordinate maps $t_i$'s are $G$-invariant, so the argument generalizes to the equivariant setting.
\end{proof}

\subsection{Bounds}

The proof of the following lemma follows from \Cref{lemma: eq def symmetrized}.

\begin{lemma}
\label{lemma: eq def mono-symmetrized}
Let $p \colon E\to B$ be a fibration, and let $U$ be a $G$-invariant (not necessarily open) subset of $E \times_B E$ containing $\Delta E$. 
If there exists a $G$-equivariant section $\sigma$ of $\Pi \colon E^I_B\to E\times_B E$ over $U$ with $\sigma(e,e) = \mathrm{c}_e$ for all $e \in E$, then $\Delta E$ is a strong $G$-deformation retract of $U$ in $E \times_B E$, that is, there exists a $G$-homotopy $H \colon U \times I \to E \times_B E$ such that 
$$
    H(e_1,e_2,0) = (e_1,e_2), \quad
    H(e_1,e_2,1) \in \Delta E, \quad \text{and} \quad
    H(e,e,t) = (e,e)
$$
for all $(e_1,e_2) \in U, e \in E$ and $t \in I$.
\end{lemma}

\begin{remark}
Note that if $B = \{*\}$, then the converse of \Cref{lemma: eq def mono-symmetrized} also holds, see Lemma 4.5 and the proof of Theorem 5.3 in \cite{Grantsymmtc}.
If $U$ is assumed to be open and $E \times_B E$ is a perfectly normal Hausdorff space, then the existence of such a homotopy implies that the inclusion $\Delta E \hookrightarrow E \times_B E$ is a closed $\Z_2$-cofibration, see \cite{3550166}.
Hence, the converse of \Cref{lemma: eq def mono-symmetrized} is similar in form to \Cref{question: existence of monoidal section}, with the latter considered in the $\Z_2$-equivariant setting.
\end{remark}

\begin{theorem}
\label{thm: cohomological lower bound on mon-sym-para-tc}
Let $p\colon E\to B$ be a fibration. 
If there exists relative cohomology classes $v_1,\dots, v_k \in H^*(SP^2_B(E), dE_B; R)$ (for any commutative ring $R$) such that 
$$
    0 \neq v_1\smile \dots \smile v_k \in H^*(SP^2_B(E), dE_B; R)
$$
then $\TC^{M,\, \Sigma}[p\colon E\to B]\geq k$.    
\end{theorem}

\begin{proof}
Suppose that $\TC^{M,\, \Sigma}[p\colon E\to B]\geq k$.
Then, by \Cref{lemma: eq def mono-symmetrized}, there exists a $G$-invariant open cover $\{U_1,\dots,U_k\}$ of $E \times_B E$ such that, for each $i$ there exists a $G$-homotopy $H_i \colon U_i \times I \to E \times_B E$ satisfying  
$$
    H_i(e_1,e_2,0) = (e_1,e_2), \quad
    H_i(e_1,e_2,1) \in \Delta E, \quad \text{and} \quad
    H_i(e,e,t) = (e,e)
$$
for all $(e_1,e_2) \in U_i, e \in E$ and $t \in I$.

Put $\overline{U}_i=\rho(U_i)$.
Since $H_i$ is equivariant, it induces a homotopy $\overline H_i \colon \overline{U}_i \times I \to SP^2_B(E)$ satisfying
\[
    \overline H_i(z,0)=z, \quad
    \overline H_i(z,1)\in dE_B, \quad
    \overline H_i(a,t)=a \qquad
    (z \in \overline{U}_i,\, a\in dE_B,\, t\in I).
\]
Consider the long exact sequence of the triple $dE_B \subseteq \overline{U}_i\subseteq SP^2_B(E)$
\[
\begin{tikzcd}[column sep=small]
\cdots
\arrow[r]
&
H^*(SP^2_B(E),\overline{U}_i;R)
\arrow[r, "\lambda_i"]
&
H^*(SP^2_B(E),dE_B;R)
\arrow[r, "\mu_i"]
&
H^*(\overline{U}_i,dE_B;R)
\arrow[r]
&
\cdots
\end{tikzcd}
\]
We claim that $\mu_i=0$.
Indeed, the homotopy $\overline H_i \colon \overline{U}_i \times I \to SP^2_B(E)$ is fixed on $dE_B$ and deforms the inclusion of $\overline U_i$ to a map with image in $dE_B$.
Hence the inclusion of pairs
\[
    (\overline{U}_i,dE_B) \hookrightarrow (SP^2_B(E),dE_B)
\]
is homotopic, as a map of pairs, to a map of pairs which factors through $(dE_B,dE_B)$.
Since $H^*(dE_B,dE_B)=0$, the induced map
\[
    \mu_i \colon H^*(SP^2_B(E),dE_B;R) \longrightarrow H^*(\overline U_i,dE_B;R)
\]
is zero. 
In particular, $\mu_i(v_i)=0$.
Hence, there exist classes $\widetilde{v}_i \in H^*(SP^2_B(E),\overline{U}_i)$ such that $\lambda_i(\widetilde{v}_i) = v_i$.
Since 
$$
    \widetilde{v}_1 \smile \dots \smile \widetilde{v}_k
        \in H^*(SP^2_B(E),\cup_{i=1}^k\overline{U}_i;R)
        = H^*(SP^2_B(E),SP^2_B(E);R) 
        = 0,
$$
by the naturality of the cup products, it follows that $v_1 \smile \dots \smile v_k = 0$. 
This is a contradiction.
\end{proof}

\section{Computations for Fadell--Neuwirth fibrations}
\label{sec: Fadell-Neuwirth fibrations}

In this section, we establish estimates for the symmetrized parametrized topological complexity of Fadell--Neuwirth fibrations \cite{FadellNeuwirth}. 
These fibrations provide a natural topological framework for motion planning in which $n$ ordered robots move in $\R^d$ while avoiding $m$ ordered obstacles, whose positions serve as the parameter of the problem.
A parametrized motion planning algorithm therefore takes as input an obstacle configuration, along with the initial and final configurations of the robots lying in that same fibre, and outputs a continuous, collision-free motion that respects the obstacles at every time instant. 
Concretely, the Fadell--Neuwirth fibration
\[
    p \colon F(\R^d,n+m)\to F(\R^d,m) 
        \quad \text{defined by} \quad 
        p(x_1,\dots,x_n,o_1,\dots,o_m)=(o_1,\dots,o_m)
\]
models this setup, where the first $n$ coordinates correspond to the robots and the last $m$ to the obstacles.
For a fixed obstacle configuration $(o_1,\dots,o_m) \in F(\mathbb{R}^d,m)$, the relevant state space for the robots is the fibre $F(\mathbb{R}^d \setminus \mathcal{O}_m,n)$, where $\mathcal{O}_m = \{o_1,\dots,o_m\}$. 
The ordinary parametrized topological complexity measures the minimal discontinuity in any algorithm that continuously adapts to changing obstacle environments; the symmetrized variant, which is our main object of study, additionally requires that the motion from an initial configuration to a final configuration is compatible with the reverse motion, i.e. the algorithm is invariant under path reversal.

The parametrized topological complexity of these fibrations was computed in \cite{farber-para-tc} and \cite{PTCcolfree}. 
In particular, these works establish the following result.

\begin{theorem}[{\cite[Theorem 9.1]{farber-para-tc}} and {\cite[Theorem 4.1]{PTCcolfree}}]
\label{thm: ptc-Fadell_Neuwirth}
Suppose $n\geq 1$, $m\geq 2$ and $d \geq 2$. 
Then
\[
    \TC[p \colon F(\R^d,n+m)\to F(\R^d, m)] =
\begin{cases}
    2n+m-1, & \text{if $d$ is odd},\\
    2n+m-2 & \text{if $d$ is even}.
\end{cases}
\]
\end{theorem}

In \cite{Stasheff-fibre-spaces}, Stasheff established that if the base and fibre of a fibration are homotopy equivalent to CW-complexes, then so is the total space. 
Adapting Stasheff’s argument, we prove the following lemma.

\begin{lemma}
Let $p \colon E \to B$ be a Hurewicz fibration with fibre $F$.
If $B$ and $F$ have the homotopy type of CW-complexes, then $E$ also has the homotopy type of a CW-complex, and
\[
    \hdim(E) \le \hdim(B) + \hdim(F).
\]
\end{lemma}

\begin{proof}
Let $b = \hdim(B)$ and $f = \hdim(F)$. 
If $b = \infty$ or $f = \infty$, the inequality is immediate, so assume $b, f < \infty$. 
By replacing $p$ with its pullback along a homotopy equivalence,  we may assume without loss of generality that $B$ is a CW-complex of dimension $b$.

Following Stasheff \cite{Stasheff-fibre-spaces}, let $E_w = \bigcup_\alpha p^{-1}(e_\alpha)$ denote $E$ equipped with the weak topology determined by the inverse images of the closed cells $e_\alpha$ of $B$, and let $E_w^{(n)} = p^{-1}(B^{(n)})$ be the restriction over the $n$-skeleton $B^{(n)}$. 
By \cite[Proposition (2)]{Stasheff-fibre-spaces}, $E \simeq E_w = E_w^{(b)}$.

Choose a CW-complex $F$ with $\dim(F) = f$ representing the homotopy type of the fibres of $p$.
We show by induction on $n$ that $E_w^{(n)}$ is homotopy equivalent to a CW-complex $\mathcal{E}^{(n)}$ of dimension at most $n + f$.
For $n = 0$, $B^{(0)}$ is discrete and $E_w^{(0)} \simeq \mathcal{E}^{(0)} = \bigsqcup_\alpha F$ has dimension $f$.
For $1 \le n \le b$, the proof of \cite[Propositions~(0) and~(1)]{Stasheff-fibre-spaces} shows that if $E_w^{(n-1)} \simeq \mathcal{E}^{(n-1)}$ with $\dim(\mathcal{E}^{(n-1)}) \le (n-1) + f$, then $E_w^{(n)}$ has the homotopy type of a CW-complex
\[
    \mathcal{E}^{(n)} 
        = \left(\bigsqcup_\alpha e^n_\alpha \times F\right) \cup_\varphi \mathcal{E}^{(n-1)}
\]
obtained by attaching the product cells of $\bigsqcup_\alpha e^n_\alpha \times F$ along $\bigsqcup_\alpha S^{n-1}_\alpha \times F$ via a cellular map $\varphi$. Since the relative cells $e^n_\alpha \times e^k_F$ have dimension $n + k \le n + f$, it follows that
\[
    \dim(\mathcal{E}^{(n)}) 
        \le \max\bigl(\dim(\mathcal{E}^{(n-1)}),\, n + f\bigr) 
        \le \max\bigl((n-1) + f,\, n + f\bigr)
        = n + f.
\]
Taking $n = b$ gives $E \simeq E_w = E_w^{(b)} \simeq \mathcal{E}^{(b)}$ with $\dim(\mathcal{E}^{(b)}) \le b + f$, whence $\hdim(E) \le b + f$.
\end{proof}

We now adapt the proof of the previous lemma to estimate the equivariant homotopy dimension of the fibered product $E \times_B E$.

\begin{lemma}
\label{lemma: equivariant homotopy dimension of the fibre product}
Let $p \colon E \to B$ be a Hurewicz fibration with fibre $F$, and let $G = \mathbb{Z}_2$ act on the fibred product $E \times_B E$ by switching coordinates. 
If $B$ and $F$ have the homotopy type of CW-complexes, then $E \times_B E$ has the $G$-homotopy type of a $G$-CW complex, and
\[
    \hdim_G(E \times_B E) \leq \hdim(B) + 2\hdim(F).
\]
\end{lemma}

\begin{proof}
Let $b = \hdim(B)$ and $f = \hdim(F)$. 
If $b = \infty$ or $f = \infty$, the inequality is immediate, so assume $b, f < \infty$. 
By pulling $p$ back along a homotopy equivalence from a CW-complex of dimension $b$, we may assume that $B$ is a CW-complex of dimension $b$, see the proof of \Cref{prop: fibrewise homotopy invariance of sym-para-tc}.

Choose a simplicial complex $K$ of dimension $f$ with $|K| \simeq F$ (\cite[Theorem 2C.5]{Hatcher:478079}).
The simplicial Cartesian product $K\times K$, as defined in \cite[Chapter~II.8]{MR50886}, is a simplicial complex whose geometric realization satisfies $F\times F \simeq |K\times K|$.
The action of $G$ on $K\times K$ is simplicial, with the diagonal subcomplex $\Delta K\subseteq K\times K$ as its fixed-point subcomplex. 
Consequently, the fibre $F\times F$ of the Hurewicz fibration $q = p \times_B p \colon E \times_B E \to B$ is $G$-homotopy equivalent to the $G$-CW complex $\mathcal{F} = |K \times K|$ of dimension $2f$. 
Since $G$ acts trivially on $B$, the equivariant Stasheff--Sch\"on theorem (\cite[Theorem 6.1]{Wanereqfib} and \cite[Theorem 3.5.2]{MayParaHomoTheory}) implies that $E \times_B E$ has the $G$-homotopy type of a $G$-CW complex. 

Following Stasheff~\cite{Stasheff-fibre-spaces}, let $(E \times_B E)_w = \bigcup_\alpha q^{-1}(e_\alpha)$ denote $E \times_B E$ equipped with the weak topology determined by the inverse images of the closed cells $e_\alpha$ of $B$, and let $(E \times_B E)_w^{(n)} = q^{-1}(B^{(n)})$ be the restriction over the $n$-skeleton $B^{(n)}$. 
Note that the $G$-fixed-point set of $E \times_B E$ is the diagonal $\Delta E \cong E$, and the restriction of $i \colon (E \times_B E)_w \to E \times_B E$ to $G$-fixed points is $i_E \colon E_w \to E$. 
By \cite[Proposition~(2)]{Stasheff-fibre-spaces} applied to $q$ and to $p$, both $i$ and $i^G = i_E$ are homotopy equivalences; hence $i \colon (E \times_B E)_w \to E \times_B E$ is a $G$-homotopy equivalence by Bredon's equivariant Whitehead theorem (\cite{MR214062} and \cite[Theorem 5.3]{MR345103}).

We show by induction on $n$ that $(E \times_B E)_w^{(n)}$ is $G$-homotopy equivalent to a $G$-CW complex $\mathcal{E}^{(n)}$ of dimension at most $n + 2f$. 
For $n = 0$, $B^{(0)}$ is discrete and $(E \times_B E)_w^{(0)} \simeq_G \mathcal{E}^{(0)} = \bigsqcup_\alpha \mathcal{F}$ has dimension $2f$. 
For $1 \le n \le b$, if $\phi_\alpha \colon e^n_\alpha \times |K| \to p^{-1}(e^n_\alpha)$ are fibre homotopy trivializations for $p$ over the $n$-cells of $B$ in \cite[Proposition~(1)]{Stasheff-fibre-spaces}, their fibrewise squares $\phi_\alpha \times_{e^n_\alpha} \phi_\alpha \colon e^n_\alpha \times \mathcal{F} \to q^{-1}(e^n_\alpha)$ are $G$-equivariant fibre homotopy trivializations for $q$, whose restrictions to $G$-fixed points recover $\phi_\alpha$. 
Consequently, the gluing map
\[
    \Phi \colon \left(\bigsqcup_\alpha e^n_\alpha \times \mathcal{F}\right) \cup_\nu (E \times_B E)_w^{(n-1)} \longrightarrow (E \times_B E)_w^{(n)}
\]
of \cite[Proposition~(1)]{Stasheff-fibre-spaces} is $G$-equivariant and restricts on $G$-fixed points to the corresponding gluing map for $p$. 
By the proof of \cite[Proposition~(1)]{Stasheff-fibre-spaces} applied to $q$ and to $p$, both $\Phi$ and $\Phi^G$ are homotopy equivalences, so $\Phi$ is a $G$-homotopy equivalence. 
Using the inductive hypothesis, $(E \times_B E)_w^{(n-1)} \simeq_G \mathcal{E}^{(n-1)}$ and $G$-cellularly approximating the induced attaching map $\varphi$ on $\bigsqcup_\alpha S^{n-1}_\alpha \times \mathcal{F}$ (\cite[Theorem 3.5]{MR558178}), we obtain a $G$-CW complex
\[
    \mathcal{E}^{(n)} = \left(\bigsqcup_\alpha e^n_\alpha \times \mathcal{F}\right) \cup_\varphi \mathcal{E}^{(n-1)}
\]
with $\mathcal{E}^{(n)} \simeq_G (E \times_B E)_w^{(n)}$ and $\dim(\mathcal{E}^{(n)}) \le \max(\dim(\mathcal{E}^{(n-1)}),\, n + 2f) = n + 2f$. 
Taking $n = b$ yields $E \times_B E \simeq_G (E \times_B E)_w = (E \times_B E)_w^{(b)} \simeq_G \mathcal{E}^{(b)}$ with $\dim(\mathcal{E}^{(b)}) \le b + 2f$.
\end{proof}

We now compute the symmetrized parametrized topological complexity of the Fadell--Neuwirth fibrations. 
We set the notations, $E=F(\R^d,m+n)$, $B=F(\R^d,m)$, $F=F(\mathbb{R}^d \setminus \mathcal{O}_m,n)$, where $\mathcal{O}_m = \{o_1,\dots,o_m\}$ and Fadell--Neuwirth fibrations by $p \colon E\to B$. 
It is known that $F$ is $(d-2)$-connected.

\begin{theorem}
\label{thm: sym-para-tc of fadell-neuwirth}
Suppose $n \geq 1$, $m \geq 2$ and $d \geq 2$. 
Then
$$
    \TC^{\Sigma}[\mkern 1mu p \colon F(\R^d,n+m) \to F(\R^d,m)\mkern 1mu] =
        \begin{cases}
            2n+m-1, & \text{if $d$ is odd},\\
            \text{either }2n+m-2 \text{ or } 2n+m-1, & \text{if $d$ is even}. 
        \end{cases}
$$    
\end{theorem}

\begin{proof}
Note that we have
\begin{align*}
    \TC[\mkern 1mu p \colon F(\R^d,n+m) \to F(\R^d,m) \mkern 1mu]
        \leq \TC^{\Sigma}[\mkern 1mu p \colon F(\R^d,n+m) \to F(\R^d,m)\mkern 1mu]
\end{align*}
from \Cref{prop: para-tc leq sym-para-tc}. 
Hence, the lower bound follows from \Cref{thm: ptc-Fadell_Neuwirth}. 
Moreover, by \Cref{lemma: equivariant homotopy dimension of the fibre product}, and the facts that $\hdim(B)=(d-1)(m-1)$ (see \cite[Theorem 3.13]{B-Z}) , $\hdim(F)=n(d-1)$ (see \cite[Section 5]{G-G}), it follows that
$$
    \hdim_G(E\times_B E) 
        \leq  \hdim(B) + 2\hdim(F)
        = (d-1)(2n+m-1).
$$
Therefore, by \Cref{thm: dimension-connectivity upper bound on symmetrized TC}, we obtain
$$
    \TC^{\Sigma}[\mkern 1mu p \colon F(\R^d,n+m) \to F(\R^d,m)\mkern 1mu]
        < \frac{(d-1)(2n+m-1)+1}{(d-2)+1}
        = 2n+m-1 + \frac{1}{d-1}.
$$
Since $d \geq 2$, the desired result follows.
\end{proof}



We note that the estimate in the theorem above also holds for the monoidal parametrized topological complexity and the monoidal symmetrized parametrized topological complexity of Fadell--Neuwirth fibrations. 
Indeed, by \Cref{thm: equivalent defn of sym monoidal para tc}, we have
$$
    \TC^{\Sigma}[\mkern 1mu p \colon F(\R^d,n+m) \to F(\R^d,m)\mkern 1mu]
        = \TC^{M,\,\Sigma}[\mkern 1mu p \colon F(\R^d,n+m) \to F(\R^d,m)\mkern 1mu].
$$
Consequently,
\begin{align*}
     \TC[\mkern 1mu p \colon F(\R^d,n+m) \to F(\R^d,m)\mkern 1mu]
        & \leq \TC^{M}[\mkern 1mu p \colon F(\R^d,n+m) \to F(\R^d,m)\mkern 1mu] \\
        & \leq \TC^{M,\,\Sigma}[\mkern 1mu p \colon F(\R^d,n+m) \to F(\R^d,m)\mkern 1mu] \\
        & = \TC^{\Sigma}[\mkern 1mu p \colon F(\R^d,n+m) \to F(\R^d,m)\mkern 1mu].    
\end{align*}
Therefore, the upper bound estimate of the theorem above also applies to both the monoidal parametrized topological complexity and the monoidal symmetrized parametrized topological complexity of Fadell--Neuwirth fibrations.

Moreover, if $d$ is odd and $n(d-2)-m-1>0$, or if $d$ is even and $(n-1)(d-2)-m-2>0$, then the following equality follows from by \Cref{thm: mon-para-tc = para-tc if dim(E) < (TC[p] + 1)(k+1)-1}
$$
    \TC[\mkern 1mu p \colon F(\R^d,n+m) \to F(\R^d,m)\mkern 1mu]
        = \TC^{M}[\mkern 1mu p \colon F(\R^d,n+m) \to F(\R^d,m)\mkern 1mu].
$$

\section{Acknowledgment}
We thank Garc\'ia-Calcines for his feedback on the earlier version of this paper. It helped to improve the exposition and to correct some computations.
R.S. Arora thanks the Siemens-IISER Ph.D. fellowship for financial support and Industrial Consultancy and Sponsored Research (IC\&SR), Indian Institute of Technology Madras (SP25261613MADSTX009294), for financial assistance and hospitality during the author's one-month visit, during which part of this work was carried out.
N. Daundkar gratefully acknowledges the support of DST–INSPIRE Faculty Fellowship (Faculty Registration No. IFA24-MA218), as well as IC\&SR, Indian Institute of Technology Madras for the New Faculty Initiation Grant (RF25261395MANFIG009294).

\bibliographystyle{plain} 
\bibliography{References}

@article {I-S,
    AUTHOR = {Iwase, Norio and Sakai, Michihiro},
     TITLE = {Topological complexity is a fibrewise {L}-{S} category},
   JOURNAL = {Topology Appl.},
  FJOURNAL = {Topology and its Applications},
    VOLUME = {157},
      YEAR = {2010},
    NUMBER = {1},
     PAGES = {10--21},
      ISSN = {0166-8641,1879-3207},
   MRCLASS = {55M30 (55Q25 68T40)},
  MRNUMBER = {2556074},
MRREVIEWER = {J.\ M.\ Boardman},
       DOI = {10.1016/j.topol.2009.04.056},
       URL = {https://doi.org/10.1016/j.topol.2009.04.056},
}

@article{D-EqPTC, 
    title={Equivariant parametrized topological complexity}, 
    journal={Proceedings of the Royal Society of Edinburgh: Section A Mathematics}, 
    author={Daundkar, Navnath}, 
    VOLUME = {156},
      YEAR = {2026},
    NUMBER = {4},
     PAGES = {1685-1708},
    DOI={10.1017/prm.2024.117}
}

@article {EqTC,
    AUTHOR = {Colman, Hellen and Grant, Mark},
     TITLE = {Equivariant topological complexity},
   JOURNAL = {Algebr. Geom. Topol.},
  FJOURNAL = {Algebraic \& Geometric Topology},
    VOLUME = {12},
      YEAR = {2012},
    NUMBER = {4},
     PAGES = {2299--2316},
      ISSN = {1472-2747,1472-2739},
   MRCLASS = {55M30 (57S10)},
  MRNUMBER = {3020208},
MRREVIEWER = {Ian\ Hambleton},
       DOI = {10.2140/agt.2012.12.2299},
       URL = {https://doi.org/10.2140/agt.2012.12.2299},
}

@article {FarberTC,
    AUTHOR = {Farber, Michael},
     TITLE = {Topological complexity of motion planning},
   JOURNAL = {Discrete Comput. Geom.},
  FJOURNAL = {Discrete \& Computational Geometry. An International Journal
              of Mathematics and Computer Science},
    VOLUME = {29},
      YEAR = {2003},
    NUMBER = {2},
     PAGES = {211--221},
      ISSN = {0179-5376},
   MRCLASS = {68T40 (55M30 68U05)},
  MRNUMBER = {1957228},
      }

@article {Sva,
    AUTHOR = {\v{S}varc, Albert S.},
     TITLE = {The genus of a fiber space},
   JOURNAL = {Dokl. Akad. Nauk SSSR (N.S.)},
  FJOURNAL = {Doklady Akademii Nauk SSSR},
    VOLUME = {119},
      YEAR = {1958},
     PAGES = {219--222},
      ISSN = {0002-3264},
   MRCLASS = {55.00},
  MRNUMBER = {0102812},
MRREVIEWER = {J. Stasheff},
}

@article {sarkargpps,
    AUTHOR = {Sarkar, Soumen and Zvengrowski, Peter},
     TITLE = {On generalized projective product spaces and {D}old manifolds},
   JOURNAL = {Homology Homotopy Appl.},
  FJOURNAL = {Homology, Homotopy and Applications},
    VOLUME = {24},
      YEAR = {2022},
    NUMBER = {2},
     PAGES = {265--289},
      ISSN = {1532-0073},
   MRCLASS = {57R25 (55N10 55R25 57R20 57R42)},
  MRNUMBER = {4478131},
}

@article {gonzalezhighertc,
    AUTHOR = {Basabe, Ibai and Gonz\'{a}lez, Jes\'{u}s and Rudyak, Yuli B. and
              Tamaki, Dai},
     TITLE = {Higher topological complexity and its symmetrization},
   JOURNAL = {Algebr. Geom. Topol.},
  FJOURNAL = {Algebraic \& Geometric Topology},
    VOLUME = {14},
      YEAR = {2014},
    NUMBER = {4},
     PAGES = {2103--2124},
      ISSN = {1472-2747},
   MRCLASS = {55M30 (55R80)},
  MRNUMBER = {3331610},
MRREVIEWER = {Dirk Sch\"{u}tz},}

@Article{secat,
 Author = {Berstein, I. and Ganea, T.},
 Title = {The category of a map and of a cohomology class},
 FJournal = {Fundamenta Mathematicae},
 Journal = {Fundam. Math.},
 ISSN = {0016-2736},
 Volume = {50},
 Pages = {265--279},
 Year = {1962},
 Language = {English},
 DOI = {10.4064/fm-50-3-265-279},
 zbMATH = {3305695},
 Zbl = {0192.29302}
}

@article {Grantsymmtc,
    AUTHOR = {Grant, Mark},
     TITLE = {Symmetrized topological complexity},
   JOURNAL = {J. Topol. Anal.},
  FJOURNAL = {Journal of Topology and Analysis},
    VOLUME = {11},
      YEAR = {2019},
    NUMBER = {2},
     PAGES = {387--403},
      ISSN = {1793-5253,1793-7167},
   MRCLASS = {55M30 (55P91 55R91 55S15 55S40 68T40)},
  MRNUMBER = {3958926},
MRREVIEWER = {Marzieh\ Bayeh},
       DOI = {10.1142/S1793525319500183},
       URL = {https://doi.org/10.1142/S1793525319500183},
}

@book {CLOT,
    AUTHOR = {Cornea, Octav and Lupton, Gregory and Oprea, John and Tanr\'{e},
              Daniel},
     TITLE = {Lusternik-{S}chnirelmann category},
    SERIES = {Mathematical Surveys and Monographs},
    VOLUME = {103},
 PUBLISHER = {American Mathematical Society, Providence, RI},
      YEAR = {2003},
     PAGES = {xviii+330},
      ISBN = {0-8218-3403-5},
   MRCLASS = {55M30 (53D35 55P62 55Q25 57R17)},
  MRNUMBER = {1990857},
MRREVIEWER = {Samuel B. Smith},
}

@article {Sequential-PTC-and-related-invariants,
    AUTHOR = {Farber, Michael and Oprea, John},
     TITLE = {Sequential parametrized topological complexity and related
              invariants},
   JOURNAL = {Algebr. Geom. Topol.},
  FJOURNAL = {Algebraic \& Geometric Topology},
    VOLUME = {24},
      YEAR = {2024},
    NUMBER = {3},
     PAGES = {1755--1780},
      ISSN = {1472-2747,1472-2739},
   MRCLASS = {55M30},
  MRNUMBER = {4767888},
MRREVIEWER = {Stephan\ Mescher},
       DOI = {10.2140/agt.2024.24.1755},
       URL = {https://doi.org/10.2140/agt.2024.24.1755},
}

@article {fibrewise,
    AUTHOR = {Garc\'{\i}a-Calcines, J. M.},
     TITLE = {Formal aspects of parametrized topological complexity and its
              pointed version},
   JOURNAL = {J. Topol. Anal.},
  FJOURNAL = {Journal of Topology and Analysis},
    VOLUME = {15},
      YEAR = {2023},
    NUMBER = {4},
     PAGES = {1129--1148},
      ISSN = {1793-5253,1793-7167},
   MRCLASS = {55M30 (55R70 55U35)},
  MRNUMBER = {4689411},
       DOI = {10.1142/S1793525321500631},
       URL = {https://doi.org/10.1142/S1793525321500631},
}

@article{FG07,
  title={Symmetric motion planning},
  author={Farber, Michael and Grant, Mark},
  journal={Contemporary Mathematics},
  volume={438},
  pages={85--104},
  year={2007},
  publisher={Providence, RI: American Mathematical Society}
}

@article {Davis,
    AUTHOR = {Davis, Donald M.},
     TITLE = {Projective product spaces},
   JOURNAL = {J. Topol.},
  FJOURNAL = {Journal of Topology},
    VOLUME = {3},
      YEAR = {2010},
    NUMBER = {2},
     PAGES = {265--279},
      ISSN = {1753-8416},
   MRCLASS = {55R25 (55M30 55P15 57R42)},
  MRNUMBER = {2651360},
MRREVIEWER = {Jes\'{u}s Gonz\'{a}lez},
       DOI = {10.1112/jtopol/jtq006},
       URL = {https://doi.org/10.1112/jtopol/jtq006},
}

@article {ptcspherebundles,
    AUTHOR = {Farber, Michael and Weinberger, Shmuel},
     TITLE = {Parametrized topological complexity of sphere bundles},
   JOURNAL = {Topol. Methods Nonlinear Anal.},
  FJOURNAL = {Topological Methods in Nonlinear Analysis},
    VOLUME = {61},
      YEAR = {2023},
    NUMBER = {1},
     PAGES = {161--177},
      ISSN = {1230-3429},
   MRCLASS = {55M30},
  MRNUMBER = {4583972},
MRREVIEWER = {Marzieh\ Bayeh},
}

@article {Dold,
    AUTHOR = {Dold, Albrecht},
     TITLE = {Erzeugende der {T}homschen {A}lgebra {${\mathcal N}$}},
   JOURNAL = {Math. Z.},
  FJOURNAL = {Mathematische Zeitschrift},
    VOLUME = {65},
      YEAR = {1956},
     PAGES = {25--35},
      ISSN = {0025-5874},
   MRCLASS = {55.0X},
  MRNUMBER = {79269},
MRREVIEWER = {W. S. Massey},
       DOI = {10.1007/BF01473868},
       URL = {https://doi.org/10.1007/BF01473868},
}

@Misc{Bredon,
 Author = {Bredon, Glen E.},
 Title = {Introduction to compact transformation groups},
 Year = {1972},
 Language = {English},
 HowPublished = {Pure and {Applied} {Mathematics}, 46. {New} {York}-{London}: {Academic} {Press}. {XIII},459 p. \$ 21.00 (1972).},
 zbMATH = {3390006},
 Zbl = {0246.57017}
}

@article {PTCcolfree,
    AUTHOR = {Cohen, Daniel C. and Farber, Michael and Weinberger, Shmuel},
     TITLE = {Parametrized topological complexity of collision-free motion
              planning in the plane},
   JOURNAL = {Ann. Math. Artif. Intell.},
  FJOURNAL = {Annals of Mathematics and Artificial Intelligence},
    VOLUME = {90},
      YEAR = {2022},
    NUMBER = {10},
     PAGES = {999--1015},
      ISSN = {1012-2443,1573-7470},
   MRCLASS = {55M30 (68U05 70Q05)},
  MRNUMBER = {4510496},
MRREVIEWER = {Jose\ M.\ Garc\'{\i}a Calcines},
       DOI = {10.1007/s10472-022-09801-6},
       URL = {https://doi.org/10.1007/s10472-022-09801-6},
}

@article{crabb2023fibrewise,
  title={Fibrewise topological complexity of sphere and projective bundles},
  author={Crabb, MC},
  journal={arXiv preprint arXiv:2305.12836},
  year={2023}
}

@article {farber-para-tc,
    AUTHOR = {Cohen, Daniel C. and Farber, Michael and Weinberger, Shmuel},
     TITLE = {Topology of parametrized motion planning algorithms},
   JOURNAL = {SIAM J. Appl. Algebra Geom.},
  FJOURNAL = {SIAM Journal on Applied Algebra and Geometry},
    VOLUME = {5},
      YEAR = {2021},
    NUMBER = {2},
     PAGES = {229--249},
      ISSN = {2470-6566},
   MRCLASS = {55M30 (55R80 55S40 70Q05)},
  MRNUMBER = {4272901},
MRREVIEWER = {Tan\ Zhang},
       DOI = {10.1137/20M1358505},
       URL = {https://doi.org/10.1137/20M1358505},
}

@article {FadellNeuwirth,
    AUTHOR = {Fadell, Edward and Neuwirth, Lee},
     TITLE = {Configuration spaces},
   JOURNAL = {Math. Scand.},
  FJOURNAL = {Mathematica Scandinavica},
    VOLUME = {10},
      YEAR = {1962},
     PAGES = {111--118},
      ISSN = {0025-5521,1903-1807},
   MRCLASS = {55.60},
  MRNUMBER = {141126},
MRREVIEWER = {R.\ H.\ Fox},
       DOI = {10.7146/math.scand.a-10517},
       URL = {https://doi.org/10.7146/math.scand.a-10517},
}

@book {JPMay,
    AUTHOR = {May, J. P.},
     TITLE = {A concise course in algebraic topology},
    SERIES = {Chicago Lectures in Mathematics},
 PUBLISHER = {University of Chicago Press, Chicago, IL},
      YEAR = {1999},
     PAGES = {x+243},
      ISBN = {0-226-51182-0; 0-226-51183-9},
   MRCLASS = {55-02 (18-02 57-02)},
  MRNUMBER = {1702278},
MRREVIEWER = {R.\ M.\ Vogt},
}

@article {minowa2024parametrized,
    AUTHOR = {Minowa, Yuki},
     TITLE = {Parametrized topological complexity of spherical fibrations
              over spheres},
   JOURNAL = {Math. Z.},
  FJOURNAL = {Mathematische Zeitschrift},
    VOLUME = {311},
      YEAR = {2025},
    NUMBER = {1},
     PAGES = {Paper No. 1, 35},
      ISSN = {0025-5874,1432-1823},
   MRCLASS = {55M30 (55S40)},
  MRNUMBER = {4920024},
MRREVIEWER = {Stephan\ Mescher},
       DOI = {10.1007/s00209-025-03794-8},
       URL = {https://doi.org/10.1007/s00209-025-03794-8},
}

@article {GrantPTC,
    AUTHOR = {Grant, Mark},
     TITLE = {Parametrised topological complexity of group epimorphisms},
   JOURNAL = {Topol. Methods Nonlinear Anal.},
  FJOURNAL = {Topological Methods in Nonlinear Analysis},
    VOLUME = {60},
      YEAR = {2022},
    NUMBER = {1},
     PAGES = {287--303},
      ISSN = {1230-3429},
   MRCLASS = {55M30 (20J06 55P20)},
  MRNUMBER = {4524869},
MRREVIEWER = {Sam\ Hughes},
}

@article {petar-trivial-fibrations,
    AUTHOR = {Pave\v si\'c, Petar},
     TITLE = {A note on trivial fibrations},
   JOURNAL = {Glas. Mat. Ser. III},
  FJOURNAL = {Glasnik Matemati\v cki. Serija III},
    VOLUME = {46(66)},
      YEAR = {2011},
    NUMBER = {2},
     PAGES = {513--519},
      ISSN = {0017-095X,1846-7989},
   MRCLASS = {55R35},
  MRNUMBER = {2855030},
MRREVIEWER = {Samuel\ B.\ Smith},
       DOI = {10.3336/gm.46.2.19},
       URL = {https://doi.org/10.3336/gm.46.2.19},
}

@book{munkres2000topology,
    title={Topology},
    edition={2nd},
    author={Munkres, James R},
    year={2000},
    publisher={Prentice Hall}
}

@book {borsuk-retracts,
    AUTHOR = {Borsuk, Karol},
     TITLE = {Theory of retracts},
    SERIES = {Monografie Matematyczne [Mathematical Monographs]},
    VOLUME = {Tom 44},
 PUBLISHER = {Pa\'nstwowe Wydawnictwo Naukowe, Warsaw},
      YEAR = {1967},
     PAGES = {251},
   MRCLASS = {54.60 (55.00)},
  MRNUMBER = {216473},
MRREVIEWER = {S.\ -T\ Hu},
}

@book {daverman-manifolds,
    AUTHOR = {Daverman, Robert J.},
     TITLE = {Decompositions of manifolds},
    SERIES = {Pure and Applied Mathematics},
    VOLUME = {124},
 PUBLISHER = {Academic Press, Inc., Orlando, FL},
      YEAR = {1986},
     PAGES = {xii+317},
      ISBN = {0-12-204220-4},
   MRCLASS = {57-01 (54B15)},
  MRNUMBER = {872468},
MRREVIEWER = {Du\v san\ Repov\v s},
}

@article {murayama-G-ANR,
    AUTHOR = {Murayama, Mitutaka},
     TITLE = {On {$G$}-{ANR}s and their {$G$}-homotopy types},
   JOURNAL = {Osaka J. Math.},
  FJOURNAL = {Osaka Journal of Mathematics},
    VOLUME = {20},
      YEAR = {1983},
    NUMBER = {3},
     PAGES = {479--512},
      ISSN = {0030-6126},
   MRCLASS = {57S10 (54C55)},
  MRNUMBER = {718960},
MRREVIEWER = {J.\ W.\ Jaworowski},
       URL = {http://projecteuclid.org/euclid.ojm/1200776318},
}

@book {dold-alg-top,
    AUTHOR = {Dold, Albrecht},
     TITLE = {Lectures on algebraic topology},
    SERIES = {Classics in Mathematics},
      NOTE = {Reprint of the 1972 edition},
 PUBLISHER = {Springer-Verlag, Berlin},
      YEAR = {1995},
     PAGES = {xii+377},
      ISBN = {3-540-58660-1},
   MRCLASS = {55-02 (01A75)},
  MRNUMBER = {1335915},
       DOI = {10.1007/978-3-642-67821-9},
       URL = {https://doi.org/10.1007/978-3-642-67821-9},
}

@article{Jaworowski1976,
author = {Jaworowski, Jan W.},
journal = {Mathematische Zeitschrift},
pages = {143-148},
title = {Extensions of {$G$}-maps and {E}uclidean {$G$}-retracts.},
url = {http://eudml.org/doc/172303},
volume = {146},
year = {1976},
}

@article {aguilar-gonzalez-2023motion,
    AUTHOR = {Aguilar-Guzm\'an, Jorge and Gonz\'alez, Jes\'us},
     TITLE = {Motion planning in polyhedral products of groups and a
              {F}adell-{H}usseini approach to topological complexity},
   JOURNAL = {Topol. Methods Nonlinear Anal.},
  FJOURNAL = {Topological Methods in Nonlinear Analysis},
    VOLUME = {61},
      YEAR = {2023},
    NUMBER = {1},
     PAGES = {37--57},
      ISSN = {1230-3429},
   MRCLASS = {55M30 (55M15)},
  MRNUMBER = {4583966},
}

@article{Inv-Para-TC,
  title={Equivariant and invariant parametrized topological complexity},
  author={Arora, Ramandeep Singh and Daundkar, Navnath},
  journal={Proceedings of the Royal Society of Edinburgh: Section A Mathematics},
  year={2026},
  pages={1–44},
  DOI={10.1017/prm.2026.10147}
}

@article{dranishnikov2014topological,
  title={Topological complexity of wedges and covering maps},
  author={Dranishnikov, Alexander},
  journal={Proceedings of the American Mathematical Society},
  volume={142},
  number={12},
  pages={4365--4376},
  year={2014}
}

@article{garcia2019note,
  title={A note on covers defining relative and sectional categories},
  author={Garc{\'\i}a-Calcines, Jose Manuel},
  journal={Topology and its Applications},
  volume={265},
  pages={106810},
  year={2019},
  publisher={Elsevier}
}

@article{carrasquel2014relative,
  title={Relative category and monoidal topological complexity},
  author={Carrasquel-Vera, Jos{\'e} Gabriel and Garcia-Calcines, JM and Vandembroucq, Lucile},
  journal={Topology and its Applications},
  volume={171},
  pages={41--53},
  year={2014},
  publisher={Elsevier}
}

@article {secat-and-relcat-II,
    AUTHOR = {Doeraene, Jean-Paul and El Haouari, Mohammed},
     TITLE = {When does secat equal relcat ?},
   JOURNAL = {Bull. Belg. Math. Soc. Simon Stevin},
  FJOURNAL = {Bulletin of the Belgian Mathematical Society. Simon Stevin},
    VOLUME = {20},
      YEAR = {2013},
    NUMBER = {5},
     PAGES = {769--776},
      ISSN = {1370-1444,2034-1970},
   MRCLASS = {55M30},
  MRNUMBER = {3160587},
MRREVIEWER = {Dae-Woong\ Lee},
       URL = {http://projecteuclid.org/euclid.bbms/1385390762},
}

@article {secat-and-relcat-I,
    AUTHOR = {Doeraene, Jean-Paul and El Haouari, Mohammed},
     TITLE = {Up-to-one approximations of sectional category and topological
              complexity},
   JOURNAL = {Topology Appl.},
  FJOURNAL = {Topology and its Applications},
    VOLUME = {160},
      YEAR = {2013},
    NUMBER = {5},
     PAGES = {766--783},
      ISSN = {0166-8641,1879-3207},
   MRCLASS = {55M30 (18B30)},
  MRNUMBER = {3022743},
MRREVIEWER = {Jean-Baptiste\ Gatsinzi},
       DOI = {10.1016/j.topol.2013.02.001},
       URL = {https://doi.org/10.1016/j.topol.2013.02.001},
}

@article {I-S-erratum,
    AUTHOR = {Iwase, Norio and Sakai, Michihiro},
     TITLE = {Erratum to ``{T}opological complexity is a fibrewise {L}-{S}
              category'' [{T}opology {A}ppl. 157 (1) (2010) 10--21]
              [MR2556074]},
   JOURNAL = {Topology Appl.},
  FJOURNAL = {Topology and its Applications},
    VOLUME = {159},
      YEAR = {2012},
    NUMBER = {10-11},
     PAGES = {2810--2813},
      ISSN = {0166-8641,1879-3207},
   MRCLASS = {55M30 (55Q25 68T40)},
  MRNUMBER = {2923451},
       DOI = {10.1016/j.topol.2012.03.009},
       URL = {https://doi.org/10.1016/j.topol.2012.03.009},
}

@MISC {3550166,
    TITLE = {Neighborhood deformation retracts vs cofibrations},
    AUTHOR = {Tyrone (https://math.stackexchange.com/users/258571/tyrone)},
    HOWPUBLISHED = {Mathematics Stack Exchange},
    NOTE = {URL:https://math.stackexchange.com/q/3550166 (version: 2023-01-08)},
    EPRINT = {https://math.stackexchange.com/q/3550166},
    URL = {https://math.stackexchange.com/q/3550166}
}

@article{A-D-S,
    AUTHOR = {Arora, Ramandeep Singh and Daundkar, Navnath and Sarkar, Soumen},
    TITLE = {Sectional category with respect to group actions and sequential topological complexity of fibre bundles},
    JOURNAL = {Homology Homotopy Appl.},
    FJOURNAL = {Homology, Homotopy and Applications},
    VOLUME = {28},
    NUMBER = {3},
    YEAR = {2026},
    PAGES = {101--131},
    DOI = {10.4310/HHA.2026.v28.n3.a5},
    URL = {https://dx.doi.org/10.4310/HHA.2026.v28.n3.a5},
}

@article {eqtcprodineq,
    AUTHOR = {Gonz\'alez, Jes\'us and Grant, Mark and Torres-Giese, Enrique
              and Xicot\'encatl, Miguel},
     TITLE = {Topological complexity of motion planning in projective product spaces},
   JOURNAL = {Algebr. Geom. Topol.},
  FJOURNAL = {Algebraic \& Geometric Topology},
    VOLUME = {13},
      YEAR = {2013},
    NUMBER = {2},
     PAGES = {1027--1047},
      ISSN = {1472-2747,1472-2739},
   MRCLASS = {57R42 (55M30 68T40)},
  MRNUMBER = {3044600},
MRREVIEWER = {Yuli\ B.\ Rudyak},
       DOI = {10.2140/agt.2013.13.1027},
       URL = {https://doi.org/10.2140/agt.2013.13.1027},
}

@book {james1984topology,
    AUTHOR = {James, I. M.},
     TITLE = {General topology and homotopy theory},
 PUBLISHER = {Springer-Verlag, New York},
      YEAR = {1984},
     PAGES = {iv+248},
      ISBN = {0-387-90970-2},
   MRCLASS = {55-01 (18-01 22-01 54-01)},
  MRNUMBER = {762634},
MRREVIEWER = {Ross\ Geoghegan},
       DOI = {10.1007/978-1-4613-8283-6},
       URL = {https://doi.org/10.1007/978-1-4613-8283-6},
}

@article {B-Z,
    AUTHOR = {Blagojevi\'c, Pavle V. M. and Ziegler, G\"unter M.},
     TITLE = {Convex equipartitions via equivariant obstruction theory},
   JOURNAL = {Israel J. Math.},
  FJOURNAL = {Israel Journal of Mathematics},
    VOLUME = {200},
      YEAR = {2014},
    NUMBER = {1},
     PAGES = {49--77},
      ISSN = {0021-2172,1565-8511},
   MRCLASS = {52C17 (52B45 55P91 55R80)},
  MRNUMBER = {3219570},
MRREVIEWER = {Zhi\ L\"u},
       DOI = {10.1007/s11856-014-1006-6},
       URL = {https://doi.org/10.1007/s11856-014-1006-6},
}

@article {TC-of-configuration-spaces,
    AUTHOR = {Farber, Michael and Grant, Mark},
     TITLE = {Topological complexity of configuration spaces},
   JOURNAL = {Proc. Amer. Math. Soc.},
  FJOURNAL = {Proceedings of the American Mathematical Society},
    VOLUME = {137},
      YEAR = {2009},
    NUMBER = {5},
     PAGES = {1841--1847},
      ISSN = {0002-9939,1088-6826},
   MRCLASS = {55R80 (55M99 68T40)},
  MRNUMBER = {2470845},
MRREVIEWER = {Daniel\ C.\ Cohen},
       DOI = {10.1090/S0002-9939-08-09808-0},
       URL = {https://doi.org/10.1090/S0002-9939-08-09808-0},
}

@article {MR3665579,
    AUTHOR = {Davis, Donald M.},
     TITLE = {The symmetrized topological complexity of the circle},
   JOURNAL = {New York J. Math.},
  FJOURNAL = {New York Journal of Mathematics},
    VOLUME = {23},
      YEAR = {2017},
     PAGES = {593--602},
      ISSN = {1076-9803},
   MRCLASS = {55M30 (55M25)},
  MRNUMBER = {3665579},
MRREVIEWER = {Jose\ M.\ Garc\'ia Calcines},
       URL = {http://nyjm.albany.edu:8000/j/2017/23_593.html},
}

@article {Stasheff-fibre-spaces,
    AUTHOR = {Stasheff, James},
     TITLE = {A classification theorem for fibre spaces},
   JOURNAL = {Topology},
  FJOURNAL = {Topology. An International Journal of Mathematics},
    VOLUME = {2},
      YEAR = {1963},
     PAGES = {239--246},
      ISSN = {0040-9383},
   MRCLASS = {55.50},
  MRNUMBER = {154286},
MRREVIEWER = {E.\ Fadell},
       DOI = {10.1016/0040-9383(63)90006-5},
       URL = {https://doi.org/10.1016/0040-9383(63)90006-5},
}

@article {Wanereqfib,
    AUTHOR = {Waner, Stefan},
     TITLE = {Equivariant fibrations and transfer},
   JOURNAL = {Trans. Amer. Math. Soc.},
  FJOURNAL = {Transactions of the American Mathematical Society},
    VOLUME = {258},
      YEAR = {1980},
    NUMBER = {2},
     PAGES = {369--384},
      ISSN = {0002-9947,1088-6850},
   MRCLASS = {55P99 (55R05 57S15)},
  MRNUMBER = {558179},
MRREVIEWER = {A.\ Dold},
       DOI = {10.2307/1998062},
       URL = {https://doi.org/10.2307/1998062},
}

@book {MayParaHomoTheory,
    AUTHOR = {May, J. P. and Sigurdsson, J.},
     TITLE = {Parametrized homotopy theory},
    SERIES = {Mathematical Surveys and Monographs},
    VOLUME = {132},
 PUBLISHER = {American Mathematical Society, Providence, RI},
      YEAR = {2006},
     PAGES = {x+441},
      ISBN = {978-0-8218-3922-5; 0-8218-3922-5},
   MRCLASS = {55P42 (19L99 55N20 55N22 55P91)},
  MRNUMBER = {2271789},
MRREVIEWER = {A.\ A.\ Ranicki},
       DOI = {10.1090/surv/132},
       URL = {https://doi.org/10.1090/surv/132},
}

@book {MR50886,
    AUTHOR = {Eilenberg, Samuel and Steenrod, Norman},
     TITLE = {Foundations of algebraic topology},
 PUBLISHER = {Princeton University Press, Princeton, NJ},
      YEAR = {1952},
     PAGES = {xv+328},
   MRCLASS = {56.0X},
  MRNUMBER = {50886},
MRREVIEWER = {H.\ Cartan},
}

@book {MR214062,
    AUTHOR = {Bredon, Glen E.},
     TITLE = {Equivariant cohomology theories},
    SERIES = {Lecture Notes in Mathematics},
    VOLUME = {No. 34},
 PUBLISHER = {Springer-Verlag, Berlin-New York},
      YEAR = {1967},
     PAGES = {vi+64 pp. (not consecutively paged)},
   MRCLASS = {55.40 (22.00)},
  MRNUMBER = {214062},
MRREVIEWER = {N.\ Stein},
}

@article {MR345103,
    AUTHOR = {Matumoto, Takao},
     TITLE = {On {$G$}-{${\rm CW}$} complexes and a theorem of {J}. {H}.
              {C}. {W}hitehead},
   JOURNAL = {J. Fac. Sci. Univ. Tokyo Sect. IA Math.},
  FJOURNAL = {Journal of the Faculty of Science. University of Tokyo.
              Section IA. Mathematics},
    VOLUME = {18},
      YEAR = {1971},
     PAGES = {363--374},
      ISSN = {0040-8980},
   MRCLASS = {55D20},
  MRNUMBER = {345103},
MRREVIEWER = {G.\ E.\ Bredon},
}

@article {MR558178,
    AUTHOR = {Waner, Stefan},
     TITLE = {Equivariant homotopy theory and {M}ilnor's theorem},
   JOURNAL = {Trans. Amer. Math. Soc.},
  FJOURNAL = {Transactions of the American Mathematical Society},
    VOLUME = {258},
      YEAR = {1980},
    NUMBER = {2},
     PAGES = {351--368},
      ISSN = {0002-9947,1088-6850},
   MRCLASS = {55P99 (55R05 57S15)},
  MRNUMBER = {558178},
MRREVIEWER = {A.\ Dold},
       DOI = {10.2307/1998061},
       URL = {https://doi.org/10.2307/1998061},
}

@book{Hatcher:478079,
  author     = {Hatcher, Allen},
  title         = "{Algebraic topology}",
  address       = {Cambridge},
  year       = "2000",
  publisher   = {Cambridge Univ. Press}
}

@article {G-G,
    AUTHOR = {Gonz\'alez, Jes\'us and Grant, Mark},
     TITLE = {Sequential motion planning of non-colliding particles in
              {E}uclidean spaces},
   JOURNAL = {Proc. Amer. Math. Soc.},
  FJOURNAL = {Proceedings of the American Mathematical Society},
    VOLUME = {143},
      YEAR = {2015},
    NUMBER = {10},
     PAGES = {4503--4512},
      ISSN = {0002-9939,1088-6826},
   MRCLASS = {55R80 (55M30 55S40 68T40)},
  MRNUMBER = {3373948},
MRREVIEWER = {Yuli\ B.\ Rudyak},
       DOI = {10.1090/proc/12443},
       URL = {https://doi.org/10.1090/proc/12443},
}
\end{document}